\documentclass[11pt,reqno]{amsart}
\usepackage[letterpaper,margin=1in,footskip=0.25in]{geometry}
\usepackage{mathrsfs}
\usepackage{amssymb}
\usepackage{microtype}
\usepackage{mathtools}
\usepackage{enumitem}
\PassOptionsToPackage{dvipsnames}{xcolor}
\usepackage{tikz-cd}
\PassOptionsToPackage{pdfusetitle,pagebackref,colorlinks}{hyperref}
\usepackage{bookmark}
\hypersetup{citecolor=OliveGreen,linkcolor=Mahogany,urlcolor=Plum}
\usepackage[hyphenbreaks]{breakurl}

\newtheorem{theorem}{Theorem}[section]
\newtheorem{lemma}[theorem]{Lemma}
\newtheorem{proposition}[theorem]{Proposition}

\newtheorem{conjecture}[theorem]{Conjecture}

  \newtheorem{innercustomprop}{Proposition}
\newenvironment{customprop}[1]
  {\renewcommand\theinnercustomprop{#1}\innercustomprop}
  {\endinnercustomprop}
  
  \theoremstyle{theorem} \newtheorem*{mainth}{Main Theorem}

\theoremstyle{definition}

\newtheorem*{question*}{Question}

\theoremstyle{remark}
\newtheorem{remark}[theorem]{Remark}

\newtheoremstyle{cited}{}{}{\itshape}{}{\bfseries}{\bfseries .}{ }{\thmname{#1} \thmnumber{#2} \thmnote{\normalfont#3}}
\theoremstyle{cited}

\newcommand{\fa}{\mathfrak{a}}
\newcommand{\fm}{\mathfrak{m}}
\newcommand{\Af}{\mathbb{A}}

\newcommand\NN{\protect\mathbf{N}}
\newcommand\QQ{\protect\mathbf{Q}}
\newcommand\RR{\protect\mathbf{R}}
\newcommand\ZZ{\protect\mathbf{Z}}

\newcommand\PP{\protect\mathbb{P}}

\newcommand\sA{\mathscr{A}}

\newcommand\cA{\mathcal{A}}

\newcommand\cI{\mathcal{I}}
\newcommand\cJ{\mathcal{J}}
\newcommand\cK{\mathcal{K}}

\newcommand\cO{\mathcal{O}}

\DeclareMathOperator{\Fitt}{Fitt}

\DeclareMathOperator{\hs}{e}
\DeclareMathOperator{\Jac}{Jac}

\DeclareMathOperator\Int{Int}
\DeclareMathOperator\init{in}
\DeclareMathOperator\lct{lct}

\DeclareMathOperator\ord{ord}
\DeclareMathOperator\red{red}
\DeclareMathOperator\Sing{Sing}
\DeclareMathOperator{\Spec}{Spec}
\DeclareMathOperator\Supp{Supp}
\DeclareMathOperator{\toric}{toric}
\DeclareMathOperator{\Val}{Val}
\DeclareMathOperator{\Vol}{Vol}

\DeclareMathOperator\Cosupp{Cosupp}
\DeclareMathOperator\Exc{Exc}
\newcommand\la{\lambda}
\newcommand\fb{\mathfrak{b}}

\newcommand\ab{\mathfrak{a}_{\bullet}}
\newcommand\fp{\mathfrak{p}}
\DeclareMathOperator{\length}{length}

\DeclareMathOperator\vol{vol}
\DeclareMathOperator\nvol{\widehat{\vol}}

\begin{document}
\title[Existence of Valuations with Smallest Normalized Volume]{Existence of Valuations with Smallest Normalized Volume}

\author{Harold Blum}
\thanks{This material is based upon work supported by the National Science
Foundation under Grant No.\ DMS-0943832}
\address{Department of Mathematics, University of Michigan, Ann Arbor, MI
48109-1043, USA}
\email{\href{mailto:blum@umich.edu}{blum@umich.edu}}
\urladdr{\url{http://www-personal.umich.edu/~blum/}}

\makeatletter
  \hypersetup{pdfsubject=\@subjclass,pdfkeywords=\@keywords}
\makeatother

\begin{abstract}
  Li introduced the normalized volume of a valuation due to its relation to K-semistability. He conjectured that over a klt singularity there exists a valuation with smallest normalized volume. We prove this conjecture and give an explicit example to show that such a valuation need not be divisorial. 
\end{abstract}

\maketitle

\section{Introduction}\label{s:intro}

 Fix a variety $X$ of dimension $n$ and $x\in X$ a closed point. 
 Let $\Val_{X,x}$ denote the set of real valuations on $X$ with center equal to $x$. 
An element of $\Val_{X,x}$ is an $\RR$-valued valuation of the function field $K(X)$ that takes nonnegative values on $\cO_{X,x} \subseteq K(X)$ and strictly positive values on the maximal ideal of $\cO_{X,x}$. For examples,  \emph{divisorial valuations} centered at $x$ form an important class inside $\Val_{X,x}$. These valuations are determined by the order of vanishing along a prime divisor $E\subset Y$ where $Y$ is normal and  there is a proper birational morphism $f:Y \to X$ contracting $E$ to $x$. We denote such a valuation by $\ord_E \in \Val_{X,x}$. 

Li introduced the \emph{normalized volume} function\[
\nvol_{X,x} : \Val_{X,x} \longrightarrow \RR_{>0} \cup \{ +\infty\}\]
that sends a valuation $v$ to its normalized volume, denoted $\nvol(v)$ \cite{Li1}. To define the normalized volume, we recall the following. Given a valuation $v\in \Val_{X,x}$, we have valuation ideals
 \[
 \fa_m(v)_x \coloneqq  
 \{ f\in \cO_{X,x} \, \vert \, v(f) \geq m \}  \subseteq \cO_{X,x}
 \]
 for all positive integers $m$. The volume of $v$ is given by
 \[
\vol(v) \coloneqq  \underset{m\to \infty}{\limsup} \,\,\frac{ \length( \cO_{X,x}/ \fa_m(v)_x) }{m^n}
.\]
The normalized volume of $v$ is
\[
\nvol(v) \coloneqq  A_X(v)^n \vol(v),
\,\]
where $A_X(v)$ is the log discrepancy of $v$ (See Section \ref{disc}).
When $X$ has klt singularities, $A_X(v) >0$, and, thus, $\nvol(v)>0$ for all $v\in \Val_{X,x}$. Li conjectured the following. 

\begin{conjecture}[\cite{Li1}]\label{liconj}
If $X$ has klt singularities at $x$, there exists a valuation $v^\ast\in \Val_{X,x}$ that minimizes $\nvol_{X,x}$. 
Furthermore, such a minimizer $v^\ast$ is unique (up to scaling) and quasimonomial. 
\end{conjecture} 

The above conjecture holds when $x\in X$ is a smooth point. As observed in \cite{Li1}, if $x$ is a smooth point, then $\nvol_{X,x}$ is minimized at $\ord_x$, the valuation that measures order of vanishing at $x$. Thus,
\[
n^n = \nvol( \ord_x)  \leq \nvol(v)\]
for all $v\in \Val_{X,x}$.
The above observation follows from the work of 
de Fernex-Ein-Musta{\c{t}}{\u{a}}.
\begin{theorem}[\cite{DFEM}]
Let $X$ be a variety of dimension $n$ and  $x\in X$ a smooth point. If $\fa \subseteq \cO_{X,x}$ is an ideal that vanishes precisely at $x$, then
\[
n^n = \lct(\fm_x)^n \hs(\fm) \leq    \lct(\fa)^n \hs(\fa)
\]
where $\fm_x$ is the maximal ideal of $\cO_{X,x}$. 
\end{theorem}
The authors of the previous theorem were motivated by their interest in singularity theory, as well as applications to birational rigidity  \cite{DFEMR} \cite{DFEM} \cite{DF}. Li's interest in volume minimization stems from questions concerning K-semistability of Fano varities. Let $V$ be a smooth Fano variety and $C(V,-K_V): = \Spec( \oplus_{m\geq 0} H^0(V, -mK_V))$ the affine cone over $V$ with cone point $0 \in C(V,-K_V)$.  The blowup of $C(V,-K_V)$ at $0$ has a unique exceptional divisor, which we denote by $\tilde{V}$. 
 
\begin{theorem}[\cite{Li1} \cite{LiLiu}]
Let $V$ be a smooth Fano variety. The following are equivalent:
\begin{enumerate}
\item The Fano variety $V$ is $K$-semistable. 
\item The function $\nvol_{C,0}$ is minimized at $\ord_{\tilde{V}}$.
\end{enumerate}
\end{theorem}

Thus, if $V$ is K-semistable, there exists a valuation centered at $0\in C(V,-K_V)$ with smallest normalized volume. If $V$ is not K-semistable, Conjecture \ref{liconj}  implies the existence of such a valuation. We prove the following.

\begin{mainth}\label{a}
If $x\in X$ is a closed point on a klt variety, then there exists a valuation $v^\ast \in \Val_{X,x}$ that is a minimizer of $\nvol_{X,x}$.
\end{mainth}

In practice, it is rather difficult to pinpoint such a valuation $v^\ast$ satisfying the conclusion of this theorem.
For a good source of computable examples, we consider the toric setting.  
\begin{theorem}\label{b}
If  $X$ is a klt toric variety with $x\in X$ a torus invariant point, then  
\[
\inf_{v\in \Val_{X,x}^{\toric}} \nvol(v) = 
\inf_{v\in \Val_{X,x}}\nvol(v),\]
where $\Val_{X,x}^{\toric}$ denotes the set of toric valuations of $X$ with center equal to $x$. 
\end{theorem}

In Section \ref{example}, we look at a concrete example, the cone over $\PP^2$ blown up at a point. In this example, we find a quasimonomial valuation that minimizes the normalized volume function. Additionally, we show that there does not exist a divisorial volume minimizer.
 While this example is not new, it is unique in that we use entirely algebraic methods. As explained in \cite[Example 6.2]{LiXu}, examples from Sasakian geometry with irregular Sasaki-Einstein metrics will provide similar examples. Our example was looked at in \cite[Section 7]{MDS}.
\\

\noindent {\bf Sketch of the proof of the main theorem}.

In order to prove the \hyperref[a]{Main Theorem} we first take a sequence of valuations $\{v_i\}_{i\in \NN}$ such that 
\[
\lim_{i \to \infty} \nvol(v_i) = \inf_{v\in \Val_{X,x}} \nvol(v).\]
Ideally, we would would like to find a valuation $v^\ast$ that is a limit point of the collect $\{v_i \}_{i \in \NN}$  and then argue that $v^\ast$ is a minimizer of $\nvol_{X,x}$. To proceed with such an argument, one would likely need to show that $\nvol_{X,x}$ is a lower semicontinuous function on $\Val_{X,x}$. It is unclear how to prove such a statement\footnote{Li showed that if $\nvol_{X,x}$ is lower semicontinuous on $\Val_{X,x}$, then there is a minimizer of  $\nvol_{X,x}$ \cite[Corollary 3.5]{Li1}. Note that $\nvol_{X,x}(v)\coloneqq  A_X(v)^n \vol(v)$ is a product of two functions. While $A_X$ is lower semicontinuous on $\Val_{X,x}$, $\vol$ fails to be lower semicontinuous in general \cite[Proposition 3.31]{FJ}. }.

We proceed by shifting our focus. Instead of studying valuations $v\in \Val_{X,x}$, we may consider ideals $\fa \subseteq \cO_X$ that are $\fm_x$-primary. For an $\fm_x$-primary ideal, the \emph{normalized multiplicity} of $\fa$ is given by
$\lct(\fa)^n \hs(\fa)$, where \[
\lct(\fa) \coloneqq  \min_{v\in \Val_{X,x}} \frac{A_X(v)}{v(\fa)} \text{ and }
\hs(\fa)\coloneqq 
\lim_{m\to \infty} \frac{ \length (\cO_{X}/ \fa^m)}{m^n /n!},
\]
where the above invariants are the \emph{log canonical threshold} and \emph{Hilbert-Samuel multiplicity}.  

We can also define similar invariants for graded sequnces of $\fm_x$-primary ideals. Note that a graded sequence of ideals on $X$ is a sequence of ideals $\fa_\bullet= \{ \fa_m\}_{m\in \NN}$ such that $\fa_{m} \cdot \fa_n \subseteq \fa_{m+n}$ for all $m,n\in \NN$. The following proposition relates minimizing the normalized volume function to minimizing the normalized multiplicity.
\begin{customprop}{\ref{relate}}[\cite{Liu}]
If $x\in X$ is a closed point on a klt variety, then
\begin{equation}\label{infequate}
\inf_{v\in \Val_{X,x}} \nvol(v)
= 
\inf_{\ab \, \fm_x \text{-primary}   } \lct(\ab)^n \hs(\ab)
=
\inf_{\fa\, \fm_x \text{-primary}  }\lct(\fa)^n \hs(\fa).
\end{equation}
\end{customprop}

While our goal is to find $v^\ast \in \Val_{X,x}$ that achieves the first infimum of Equation \ref{infequate}, we will instead find a graded sequence of $\fm_x$-primary ideals $\tilde{\fa}_\bullet$ that achieves the second infimum of the equation. After having constructed such a graded sequence $\tilde{\fa}_\bullet$, a valuation $v^\ast$ that computes $\lct(\tilde{\fa}_\bullet)$ (see Section \ref{lct}) will be a minimizer of $\nvol_{X,x}$.

To construct such a graded sequence,  we will  take our previously mentioned sequence of valuations $\{v_i \}_{i \in \NN}$. 
This gives us a collection of graded sequences of ideals $\{ \ab(v_i) \}_{i \in \NN}$. Our goal will be to find a graded sequence $\tilde{\fa}_\bullet$ that is a ``limit point'' of the previous collection. 

We recall the work of  de Fernex-Musta{\c{t}}{\u{a}}  \cite{FM},  Koll{\'a}r \cite{Which}, and de Fernex-Ein-Musta{\c{t}}{\u{a}} \cite{Smooth} \cite{Bound} on \emph{generic limits}. Given a collection of ideals  $\{ \fa_i \}_{i \in \NN}$ where $\fa_i \subset k[x_1,\ldots, x_r]$, there exists a field extension $k\subseteq K$ and an ideal $\tilde{\fa} \subset K[[x_1,\ldots,x_r]]$ that encodes information on infinitely many members of $\{ \fa_i \}_{i \in \NN}$.  We extend previous work on generic limits to find a ``limit point'' of a collection of graded sequences of ideals. 

Along the way, we will need a technical result on the rate of convergence of  $\{ \hs(\fa_m(v)) \}_m$ for a valuation $v\in \Val_{X,x}$. To perform this task, we extend the work of Ein-Lazarsfeld-Smith on approximation of valuation ideals \cite{ELS} and prove a technical, but also surprising, uniform convergence type result for the volume function.

\begin{customprop}{\ref{volumebound}}
Let  $X$ be a klt variety of dimension $n$ and $x\in X$ a closed point. For $\epsilon>0$ and constants $B,E,r \in \ZZ_{>0}$, there exists $N=N(\epsilon, B,E,r)\in \ZZ_{>0}$ such that for every valuation $v\in \Val_{X,x}$ with  $\vol(v)\leq  B$, $A_X(v) \leq E$, and $v(\fm_x)\geq \frac{1}{r}$, we have
\[
\vol(v) \leq
\frac{
\hs(\fa_m(v)) }{m^n} < \vol(v)+ \epsilon .
\text{ for all } m\geq N.\]
\end{customprop}

\noindent{\bf Structure of the Paper:} In Section \ref{prelim} we provide preliminary information on valuations, graded sequences of ideals, and log canonical thresholds. Section \ref{secels} extends \cite{ELS} to klt varieties and gives a proof of the previous proposition on the volume of a valuation. Section \ref{normal} provides information on Li's normalized volume function. Section \ref{seclimits} extends the theory of generic limits from ideals to graded sequences of ideals. Section \ref{secmain} provides a proof of the \hyperref[a]{Main Theorem}.  In \ref{seclog}, we explain that the arguments in this paper extend to the setting of log pairs. Lastly, Section \ref{sectoric} provides a proof of Theorem \ref{b} and a computation of an example of a non-divisorial volume minimizer. 

The paper also has two appendices that collect known statements that do not  explicitly appear in the literature. Appendix \ref{secmult} provides information on the behavior of the Hilbert-Samuel multiplicity and log canonical threshold in families. Appendix \ref{appcomputing} provides a proof of the existence of valuations computing log canonical thresholds on klt varieties.
\\

\noindent{\bf Acknowledgements:}
I would like to thank my advisor Mircea Musta{\c{t}}{\u{a}} for guiding my research and sharing his ideas with me. I also thank Mattias Jonsson, Chi Li, Yuchen Liu, and Chenyang Xu for comments on a previous draft of this paper. Lastly, I am grateful to David Stapleton for producing the graphic in Section \ref{example}.

\section{Preliminaries}\label{prelim}

\noindent {\bf Conventions}:
For the purpose of this paper, a variety is an irreducible, reduced, seperated scheme of finite type over a field $k$. Furthermore, we will always assume that $k$ is of characteristic $0$, algebraically closed, and uncountable. We use the convention that $\NN= \{1,2,3,\ldots\}$.

\subsection{Real Valuations}  Let $X$ be a variety and $K(X)$ denote its function field. A \emph{real valuation} of $K(X)$ is non-trivial group homomorphism
\[
v: K(X)^\times \to \RR
\]
such that  $v$ is trivial on $k$ (the base field) and
$v(f+g) \geq \min \{ v(f),v(g) \}$. We set $v(0)=+\infty$.
A valuation $v$ gives rise to a valuation ring $\cO_v \subset K(X)$, where $\cO_v \coloneqq  \{ f \in K(X)\vert v(f) \geq 0\}$. 
Note that if $v$ is a valuation of $K(X)$ and $\la\in \RR_{>0}$, scaling the outputs of  $v$ by $\la$ gives a new valuation $\la \cdot  v$.

We say that $v$ has \emph{center on} $X$ if there exists a map $\pi :\Spec( \cO_v)\to X$ as below
\[
\begin{tikzcd}
\Spec(K(X)) \arrow{r}  \arrow{d} & X  \arrow{d} \\
\Spec(\cO_v) \arrow{r} \arrow[dotted]{ur}{\pi} & \Spec(k)
\end{tikzcd}
.\]
By 
 \cite[Theorem II.4.3]{Har}, if such a map $\pi$ exists, it is necessarily unique.
Let $\zeta$ denote the unique closed point of $\Spec(\cO_v)$. If such a $\pi$ exists, we define the \emph{center} of $v$ on $X$, denoted $c_X(v)$, to be $\pi(\zeta)$. 
We let $\Val_X$ (resp., $\Val_{X,x}$)  denote the set of real valuations of $K(X)$ with center on $X$ (resp., center equal to $x$). 

Given a valuation $v\in \Val_X$ and a nonzero ideal $\fa\subseteq \cO_X$, we may evaluate $\fa$ along $v$ by setting 
\[
v(\fa) \coloneqq  \min \{ v(f) \vert f \in \fa \cdot \cO_{X,c_X(v)}\}
.\]
In the case when $X$ is affine, the above definition can be made simpler. In this case, 
\[
v(\fa) = \min\{ v(f) \vert f\in \fa (X)\}.\]
It follows from the above definition that if $\fa\subseteq \fb \subset \cO_X$ are nonzero ideals, then $v(\fa)\geq v(\fb)$. Additionally, $v(\fa) >0$ if and only if $c_x(v) \in \Cosupp(\fa)$.\footnote{ The \emph{cosupport} of an ideal $\fa \subseteq \cO_X$ is defined as $\Cosupp(\fa) \coloneqq  \Supp( \cO_X/\fa )$.}

We endow $\Val_X$ with the weakest topology such that, for every ideal $\fa$ on $X$, the map 
$\Val_X \to \RR \cup \{+\infty \}$ defined by $v\mapsto v(\fa)$ is continuous. For information on the space of valuations, see \cite{JM} and \cite{BFFU}.

\subsection{Divisorial Valuations}
Let $ E\subset Y\overset{f}{\to} X$ be a proper birational morphism, $Y$ a normal variety, and $E$ a prime divisor on $Y$. The discrete valuation ring $\cO_{Y,E}$ gives rise to a valuation $\ord_E \in \Val_X$ that sends $a \in K(X)^\times$ to the order of vanishing of $a$ along $E$. 
Note that $\ord_E \in \Val_X$ and $c_X( \ord_E)$ is the generic point of $f(E)$. 

We say that $v\in \Val_X$ is a \emph{divisorial valuation} if there exists $E$ as above and $\la\in \RR_{> 0}$ such that $v= \la \ord_E$. Divisiorial valuations are the most ``geometric'' valuations. 

\subsection{Quasimonomial Valuations}
 A \emph{quasimonomial valuation} is a valuation that becomes monomial on some birational model over $X$. Specifically, let $f:Y\to X$ be a proper birational morphism and $p \in Y$ a  closed point such that $Y$ is regular at $p$. Given a system of parameters $y_1,\ldots,y_n \in \cO_{Y, p}$ at $p$ and ${\bf \alpha} = (\alpha_1,\ldots, \alpha_n) \in \RR_{\geq 0 }^n \setminus \{ {\bf 0} \}$, we define a valuation $v_{\bf \alpha}$  as follows. 
 For $r \in \cO_{Y,p}$ we can write $r$ in $\widehat{ \cO_{Y,p}}$ as 
 $r= \sum_{\beta \in \ZZ_{\geq 0}^n } c_\beta y^\beta$, with $c_\beta \in \widehat{ \cO_{Y,p}}$
 either zero or unit. We set 
 \[
 v_{\alpha}(r) = \min \{ \langle \alpha, \beta \rangle \vert c_\beta \neq 0 \}. \]
 
 A quasimonomial valuation is a valuation that can be written in the above form. Note that in the above example, if there exists $\lambda \in \RR_{>0}$ such that $\lambda \cdot {\bf \alpha} \in \ZZ_{\geq 0}^r$, then $v_\alpha$ is a divisorial valuation. Indeed,  take a weighted blowup of $Y$ at $p$ to find the correct exceptional divisor.

\subsection{The Relative Canonical Divisor}
Let $Y\to X$ be a proper birational morphism of normal varieties. If $X$ is $\QQ$-Gorenstein, that is $K_X$ is $\QQ$-Cartier, we define the \emph{relative canonical divisor} of $f$ to be
\[
K_{Y/X} = K_{Y} - f^\ast(K_X),\]
where $K_Y$ and $K_X$ are chosen so that $f_\ast K_Y = K_X$. While $K_{Y}$ and $K_X$ are defined up to linear equivalence, $K_{Y/X}$ is a well-defined divisor. 

We say that $X$ is a \emph{klt} variety if $X$ is normal, $\QQ$-Gorenstein, and for any projective birational morphism of normal varieties $Y \to X$ the coefficients of $K_{Y/X}$ are $>-1$. Moreover, it is sufficient to check this condition on a resolution of singularities $Y\to X$ such that the exceptional locus on $Y$ is a simple normal crossing divisor. 
For further details on klt singularites and the relative canonical divisor, see \cite[Section 2.3]{KM}.

\subsection{The Log Discrepancy of a Valuation}\label{disc}
The log discrepancy of a real valuation over a smooth variety was introduced in \cite{JM} and extended to the normal case in  \cite{BFFU}.
For our purposes, we will only need to define the log discrepancy of a valuation over a $\QQ$-Gorenstein variety $X$.

 As above, let $ E\subset Y\overset{f}{\to} X$ be a proper birational morphism, $Y$ a normal variety, and $E$ a prime divisor. Additionally, we assume that $X$ is $\QQ$-Gorenstein.
 
 We first define the log discrepancy of $\ord_E$ to be 
 \[
 A_X( \ord_E) \coloneqq  1+  \text{ coefficient of $E$ in } K_{Y/X} .\]
 We define the log discrepancy for a divisorial valuations $\la \ord_E$, by setting 
 \[
 A_X( \lambda \ord_E ): = \lambda A_X(\ord_E).\]
 There is a unique way to extend $A_X$ to a lower semicontinuous function on $\Val_X$ that respects scaling \cite[Theorem 3.1]{BFFU}.
 Thus, $A_X(\la v) = \la A_X(v)$ for all $v\in \Val_{X}$ and $\la \in \RR_{>0}$. Additionally, a variety $X$ is klt  if and only if $A_X(v)> 0$ for all $v\in \Val_X$. 

\subsection{Graded Sequences of Ideals}
A graded sequence of ideals on a variety $X$ is a sequence of ideals $\ab= \{ \fa_m\}_{m\in \NN}$ such that $\fa_{m} \cdot \fa_{n} \subset \fa_{m+n}$ for all $m,n \in \NN$. To simplify exposition, we always assume that $\fa_m $ is not equal to the zero ideal for all $m\in \NN$. 

We provide two examples of graded sequences of ideals. 
\begin{enumerate}
\item Let $\fb$ be a nonzero ideal on $X$. We may define a graded sequence $\ab$ by setting $\fa_m\coloneqq  \fb^m$ for all $m\in \NN$. This example is trivial. 
\item We fix  $v\in \Val_X$ and define $\ab(v)= \{ \fa_m(v) \}_{m \in \NN}$ as follows. If $U\subseteq X$ is an open affine set such that $c_X(v) \in U$, then
\[
\fa_m (v)(U) \coloneqq \{ f\in \cO_X(U) \,\vert\, v(f) \geq m \}.\]
If $c_x(v) \notin U$, then $\fa_m(v)(U) \coloneqq  \cO_X(U)$.  
If $c_X(v)$ is a closed point $x$, we have that each ideal $\fa_m(v)$ is $\fm_x$-primary,\footnote{This is equivalent to saying that each ideal $\fa_m(v)$ vanishes only at $x$.} where $\fm_x \subseteq \cO_X$ denotes the ideal of functions vanishing at $x$.
\end{enumerate}

Given $v\in \Val_X$ and a graded sequence $\ab$, we may evaluate $\ab$ along $v$ by setting 
\[
v(\ab) \coloneqq \inf_{m \in \NN} \frac{ v(\fa_m)}{m} = \lim_{m \to \infty} \frac{v(\fa_m)}{m}.\]
See \cite[Lemma 2.3]{JM} for  a proof of the previous equality.

\subsection{Multiplicities}\label{multsection}
Let $X$ be a variety of dimension $n$ and $x\in X$ a closed point. Let $\fm_x \subseteq \cO_X$ denote the ideal of functions vanishing at $x$. We recall that 
for an $\fm_x$-primary ideal $\fa$, the \emph{Hilbert-Samuel Multiplicity} of $\fa$ is 
\[
\hs(\fa) \coloneqq  \lim_{m\to \infty} \frac{   \length ( \cO_{X,x}/\fa^m )}{m^n /n!}
.\]
If $\fa \subseteq \fb \subseteq \cO_X$ are $\fm_x$-primary ideals on $X$, then $\hs( \fa ) \geq \hs(\fb)$. Additionally, $\hs(\fa)= \hs(\overline{\fa})$ where $\overline{\fa}$ denotes the integral closure of $\fa$. 

We recall the valuative definition of the \emph{integral closure} of an ideal $\fa$ on a normal variety $X$ \cite[Example 9.6.8]{Laz}. Let $U\subset X$ affine open subset. We have
\[
\overline{ \fa}(U)\coloneqq  \{ f\in \cO_X(U) \vert w(f)  \geq w(\fa)  \text{ for all } w\in \Val_U \text{ divisorial}  \}.\]

\subsection{Volumes}

Let $\ab$ be a graded sequence of ideals with the property that  each $\fa_m$ is $\fm_x$ primary. The \emph{volume} of $\ab$ is defined as 
\[
\vol( \ab) \coloneqq  \limsup_{m \to \infty} \frac{  \length ( \cO_{X,x}/\fa_m )}{m^n /n!}
.\]
 A similar invariant is the \emph{multiplicity} of $\ab$, which is defined  as 
\[
\hs(\ab) = \lim_{m\to \infty} \frac{ \hs(\fa_m)}{m^n}
.\]
In various degrees of generality, it has been proven that 
\[
\hs(\ab) = \vol(\ab) 
\]
\hspace{-1.45 mm} \cite[Corollary C]{ELS}
\cite[Theorem 1.7]{mus}
\cite[Theorem 3.8]{NO}
\cite[Theorem 6.5]{Cut}.
In our setting, the above equality will always hold.  Additionally, by \cite[Theorem 1.1]{Cut}, we also have that
\[
\vol( \ab) \coloneqq  \lim_{m \to \infty} \frac{  \length ( \cO_{X,x}/\fa_m )}{m^n /n!}
.\]

For a valuation $v\in \Val_{X,x}$, the \emph{volume} of $v$ is given by
\[
\vol(v) \coloneqq  \hs ( \ab(v)) .\]
Note that if $\lambda \in \RR_{>0}$, then $\vol(\lambda v) = \vol(v)/\lambda^n$. 

\subsection{Log Canonical Thresholds}\label{lct}
The \emph{log canonical threshold} is an invariant of singularities that has received considerable interest in the field of birational \cite[Section 8]{SingPairs}. 
For a nonzero ideal $\fa$ on a klt variety $X$, the log canonical threshold of $\fa$ is given by
\begin{equation}\label{lcteq}
\lct(\fa) \coloneqq  \inf_{v\in \Val_X} \frac{A_X( v)}{v(\fa)}
.\end{equation}
 (We are using the convention that if $v(\fa)=0$, then $A(v)/v(\fa) = +\infty$.) Thus, $\lct(\cO_X) = +\infty$. 
  We say that a valuation $v^\ast$ computes $\lct(\fa)$ if $\lct(\fa) = A(v^\ast)/v^\ast(\fa)$.
  
The invariant satisfies the following properties. If $m\in \ZZ_{>0}$, then
\[
\lct( \fa^m) = \lct(\fa)/m.\] Additionally, if $\fa \subseteq \fb$, then
\[
\lct(\fa) \leq \lct(\fb) .\]

In Equation \ref{lcteq}, the infimum may be taken over just the set of divisorial valuations. 
 Furthermore, let $\mu: Y\to X$ be a \emph{log resolution} of $\fa$. That is to say $\mu$ is a projective birational morphism such that:
 \begin{enumerate}
 \item $Y$ is smooth,
 \item
 $\fa\cdot \cO_Y= \cO_Y(-D)$ for an effective divisor $D$ on $Y$,
 \item $\Exc(\mu)$ has pure codimension 1, and
 \item $D_{\red}+ \Exc(\mu)$ has simple normal crossing.
 \end{enumerate}
In this case, we have
\[
\lct(\fa)= \min_{i=1,\ldots,r} \frac{ A_X(\ord_{E_i})}{\ord_{E_i}(\fa)}.\]
where $D= \sum_{i=1}^r a_i E_i$ and the $E_i$ are prime. (Note that $\ord_{E_i}(\fa)=a_i$.) 
Thus, there  exists a divisorial valuation that computes $\lct(\fa)$.

For a graded sequence of ideals $\ab$ on $X$, the log canonical threshold of $\ab$ is given by
 \[
\lct( \ab) : = \lim_{m\to \infty} m\cdot \lct(\fa_m) = \sup_{m} m\cdot \lct(\fa_m) 
.\]
By \cite{JM} in the smooth case and \cite{BFFU} in full generality, we have 
\[
\lct(\ab) = \inf_{v\in \Val_X} \frac{A_X(v)}{v(\ab))} .\]
We say $v^\ast\in \Val_X$ computes $\lct(\ab)$ if $\lct(\ab) = A_X(v^\ast)/v^\ast(\ab)$. 
 Such valuations $v^\ast$ always exist (see Appendix \ref{appcomputing}).
 When $X$ is smooth, this is precisely \cite[Theorem A]{JM}. 

\section{Approximation of Valuation Ideals}\label{secels}

In this section we extend the arguments of \cite{ELS} to approximate valuation ideals on singular varieties. We will use this approximation to determine the rate of convergence of $\{ \hs(\fa_m(v))/m^n \}_{m\in \NN}$ for a fixed valuation $v$.  The main technical tool is the asymptotic multiplier ideal of a graded family of ideals. For an excellent reference on multiplier ideals, see  \cite[Ch. 9]{Laz}.

\subsection{Multiplier Ideals}

Let $\fa \subseteq \cO_X$ be a nonzero ideal on a $\QQ$-Gorenstein variety $X$. 
Consider a log resolution $\mu:Y \to X$ of $\fa$ such that $\fa \cdot \cO_Y = \cO_Y(-D)$. For a rational number $c>0$, the \emph{multiplier ideal} 
\[
\cJ(X, c\cdot \fa) \coloneqq  \mu_\ast \cO_Y(  \lceil K_{Y/X} -  cD \rceil ) \subseteq \cO_X.\]
Note that if $c$ is an integer, than $\cJ(X,c\cdot \fa) = \cJ(X,\fa^c)$. 

Alternatively, the multiplier ideal can be understood valuatively. If $X$ is an affine variety, then 
\[
\cJ(X,c\cdot \fa)(X) = \left\{ f\in \cO_X(X) \, | \, 
v(f) > c v(\fa)- A_X(v) \text{ for all $v\in \Val_X$ } \right\} . \]
When $X$ is not necessarily affine, the above criterion allows us to understand the multiplier ideal locally.

It is important to note the relationship between the log canonical threshold and the multiplier ideal. If $X$ is klt, then 
\[
\lct(\fa) = \sup\{ c \,\vert \,\cJ(X,c \cdot \fa) = \cO_X \} .\]

The following lemma provides basic  properties of multiplier ideals. The proof is left to the reader. See \cite[Proposition 9.2.32]{Laz} for the case when $X$ is smooth.
\begin{lemma}\label{mtech}
Let $X$ be a $\QQ$-Gorenstein variety and $\fa,\fb$ nonzero ideals on $X$. 
\begin{enumerate}
\item If $X$ is a klt variety, then 
\[
\fa \subseteq \cJ(X, \fa). \]
\item If $\fa \subseteq \fb$ and $c\in \QQ_{>0}$, then
\[
\cJ(X,c\cdot \fa) \subseteq \cJ(X,c\cdot \fb).
\] 
\item For rational numbers $c\geq d>0$, we have that 
\[
\cJ(X, c \cdot \fa) \subseteq \cJ(X, d\cdot \fa).\]
\end{enumerate}
\end{lemma}

Multiplier ideals satisfy the following ``subadditivity property.'' The property was first observed and proved  by Demailly-Ein-Lazarsfeld in the smooth case \cite{DEL}. The statement was extended to the singular case in \cite[Theorem 2.3]{Tak}  and \cite[Theorem 7.3.4]{Eis}. 

\begin{theorem}[Subadditivity] \label{sub}
If $X$ is a $\QQ$-Gorenstein variety, $\fa,\fb$ nonzero ideals on $X$, and $c,d\in \ZZ_{>0}$, then 
\[
\Jac_X \cdot 
\cJ(X, c\cdot (\fa \cdot \fb) )\subseteq \cJ(X, c \cdot \fa)\cdot \cJ(X, c\cdot \fb),\]
where $\Jac_X$ denotes the Jacobian ideal of $X$. 
\end{theorem}
We recall that for a variety $X$, the \emph{Jacobian} ideal of $X$ is $\Jac_X \coloneqq  \Fitt_0( \Omega_X)$, where $\Fitt_0$ denotes the $0$-th fitting ideal as in \cite[Section 20.2]{Bud}. Note that the singular locus of $X$ is equal to  $\Cosupp(\Jac_X)$.

\subsection{Asymptotic Multiplier Ideals}
Let $\ab$ be a graded sequence of ideals on a $\QQ$-Gorenstein variety $X$ and $c>0$ a rational number. We  recall the definition of the asymptotic multiplier ideal $\cJ(c\cdot \ab)$. By Lemma \ref{mtech}, we have that 
\[
\cJ \left( X, \frac{1}{p} \cdot \fa_{mp} \right) \subseteq \cJ \left( X, \frac{1}{pq} \cdot \fa_{pqm} \right) \]
for all positive integers $p,q$. From the above inclusion and Noetherianity, we conclude that
\[
\left\{ \cJ \left(X, \frac{1}{p} \cdot \fa_{pm} \right) \right\}_{p\in \NN}
\]
has a unique maximal element. The $m$-th asymptotic multiplier ideal $\cJ(X,m\cdot \ab)$ is defined to be this element. 
Like the standard multiplier ideal, the asymptotic multiplier ideal can also be understood valuatively. 
\begin{proposition}\label{val} \cite[Theorem 1.2]{BFFU} If $\ab$ is a graded sequence of ideals on a normal affine variety and $c>0$ a rational number, then
\[
\cJ(X, c \cdot \ab) = \{ f\in\cO_X(X) \, \vert \, v(f) > cv(\ab) - A_X(v) \text{ for all $v\in \Val_X$} \}.
\]
\end{proposition}

The asymptotic multiplier ideals satisfy the following property. This property will allow us to approximate valuation ideals.

\begin{proposition}\label{comp}
If $\ab$  is a graded sequence of ideals on a klt variety $X$ and $m,\ell \in \NN$, then
\[
 (\Jac_X)^\ell  \fa_{m}^\ell \subseteq
 (\Jac_X)^\ell  \fa_{m\ell} \subseteq  \cJ(m \cdot \ab)^\ell.\]
\end{proposition}

\begin{proof}
The proof is the same as the proof of \cite[Theorem 1.7]{ELS} and relies on Theorem \ref{sub}. 
\end{proof}

\subsection{The Case of Valuation Ideals}

For a valuation $v\in \Val_X$, we examine the asymptotic multiplier ideals of $\ab(v)$. 
We first prove the following elementary lemma. 
\begin{lemma}\label{self}
If $v$ is a valuation on a variety $X$, then $v(\ab(v))=1$.
\end{lemma}
\begin{proof}
Note that $v(\fa_m(v) ) \geq m$, since $\fa_m(v)$ is the ideal of functions vanishing to at least order $m$ along $v$. Next, set $\alpha \coloneqq  v(\fa_1(v))$. We have $
\fa_1(v)^{ \lceil m/\alpha \rceil} \subseteq \fa_m(v)$,
since $v(\fa_1(v)^{\lceil m/\alpha \rceil})= \alpha \lceil m/\alpha  \rceil \geq m$.
Thus,\[
v(\fa_m(v)) \leq v(\fa_1(v)^{\lceil m /\alpha \rceil}) = \alpha \lceil m /\alpha \rceil.\]
The previous two bounds combine to show \[
1\leq  \frac{ v(\fa_m(v))}{m} 
\leq
\frac{ \alpha \cdot \lceil m/\alpha \rceil}{m}
,\]
and the result follows.
\end{proof}

The following results allows us to approximate valuation ideals. In the case when $X$ is smooth and $v$ is an Abhyankhar  valuation, the theorem below is a slight strengthening of
\cite[Theorem A]{ELS}.
\begin{theorem}\label{approx}
If $X$ is a klt variety and $v\in \Val_X$ satisfying $A_X(v) < +\infty$, then  
\[
(\Jac_X)^\ell  \cdot \fa_{m}^\ell \subseteq (\Jac_X)^\ell  \cdot \fa_{m\ell} \subseteq \fa_{m-e}^\ell
\]
for every $m\geq e$,
where $\ab : = \ab(v)$ and $e\coloneqq  \lceil A_X(v) \rceil$.
\end{theorem}
\begin{proof}
By Proposition \ref{comp}, we have that
\[
(\Jac_X)^\ell  \cdot \fa_{m}^\ell \subseteq (\Jac_X)^\ell  \cdot \fa_{m\ell} \subseteq 
\cJ(X,m \cdot \ab)^\ell.
\]
Applying Proposition \ref{val} and
and Lemma \ref{self} to $\ab$ gives that
\[
\cJ(X, m \cdot \ab) \subseteq \fa_{m-e},\]
and the result follows. 
\end{proof}

\subsection{Uniform Approximation of Volumes}
Given a valuation $v\in \Val_{X}$  centered at a closed point on a $n$-dimensional variety $X$, we have  
\[
\vol(v) = \lim_{m\to \infty} \frac{\hs(\fa_m(v))}{m^n}
,\]
where $n$ is the dimension of $X$.
The following theorem provides a uniform rate of convergence for the terms in the above limit.

\begin{proposition}\label{volumebound}
Let  $X$ be a klt $n$-dimensional variety and $x\in X$ a closed point. For $\epsilon>0$ and constants $B,E,r \in \ZZ_{>0}$, there exists $N=N(\epsilon, B,E,r)$ such that for every valuation $v\in \Val_{X,x}$ with  $\vol(v)\leq  B$, $A_X(v) \leq E$, and $v(\fm_x)\geq \frac{1}{r}$, we have
\[
\vol(v) \leq
\frac{
\hs(\fa_m(v)) }{m^n} < \vol(v)+ \epsilon .
\text{ for all } m\geq N.\]
\end{proposition}

\begin{remark}
In an earlier version of this paper, we proved the following statement with the additional assumption that $x\in X$ is an isolated singularity. We are grateful to Mircea Musta{\c{t}}{\u{a}} for noticing that a modification of the original proof allows us to prove the more general statement. 
\end{remark}

\begin{proof}[Proof of Proposition \ref{volumebound}]
For any valuation $v\in \Val_{X,x}$, the first inequality is well known. Indeed, the inclusion $\fa_{m}(v)^p \subseteq \fa_{mp}(v)$ for  $m,p\in \NN$ implies that 
\[
\frac{ \hs( \fa_{mp}(v))}{(mp)^n }
\leq 
\frac{ \hs( \fa_{m}(v))}{(m)^n }.\]
Fixing $m$ and sending $p\to \infty$ gives 
\[
\vol(v) \leq 
\frac{ \hs( \fa_{m}(v))}{(m)^n }.\]

Next, fix $v\in \Val_{X,x}$ satisfying the hypotheses in the statement  of Proposition \ref{volumebound}. We have 
\[\left(
\Jac_X\right)^\ell \fa_{ m \ell }(v) \subseteq  \left( \fa_{m-e} (v) \right)^\ell \subseteq \left( \fa_{m-E} (v) \right)^\ell,\]
where $e= \lceil A_X(v) \rceil$. The first inclusion is the statement in Theorem \ref{approx}, and the second follows from the assumption that $e \leq E$. 
After replacing $m$ by $m+E$, we get
\begin{equation}\label{inclus1}
\left( \Jac_X \right)^\ell \cdot \fa_{(m+E)\ell}(v) \subseteq  \fa_{m}(v)^\ell.
\end{equation}

On the other hand, the assumption that $v(\fm_x)\geq \frac{1}{r}$ implies that 
\begin{equation}\label{inclus2}
\fm_x^{mr} \subseteq \fa_m(v).
\end{equation} 
It follows from Inclusions \ref{inclus1} and \ref{inclus2} and the valuative criterion for integral closure (Section  \ref{multsection}) that 
\begin{equation}\label{inclus3}
( \Jac_x+ \fm_x^{mr})^\ell  \fa_{(m+E)\ell}(v) \subseteq \overline{ \fa_{m}(v)^\ell}.
\end{equation}
Indeed, let $w$ be a discrete valuation of the function field of $X$ and $f$ and $g$ local sections of $\Jac_X^i$ and $\fm_x^{mrj}$, respectively, with $i+j=\ell$. 
We have $\ell \cdot  w(f) + i \cdot w( \fa_{(m+E)\ell}(v)) \geq i \cdot w(\fa_m(v)^\ell)$ and $ w(g) \geq j  \cdot w(\fa_m(v))$ by the two inclusions. 
Thus, 
\[
w(fg) = w(f)+w(g) \geq \frac{i}{\ell} \left( w(\fa_m(v)^\ell) - w(\fa_{(m+E)\ell}(v)) \right)+ \frac{j}{\ell} w(\fa_m(v)^\ell )
\]
\[
\hspace{8 mm}= w(\fa_m(v)^\ell) - w(\fa_{(m+E)\ell}(v)) 
.\]

From Inclusion \ref{inclus3} and  Teissier's Minkowski Inequality \cite{Min},
we see
\[
\hs \left(\fa_m(v) ^\ell \right)^{1/n}  \leq
\hs
  \left(
  \left( \Jac_X +\fm_{x}^{mr}  \right)^\ell
  \right)^{1/n} + \hs\left(\fa_{(m+E)\ell }(v)\right)^{1/n}.\]
  Next,  note that if $\fa$ is an $\fm_x$-primary ideal, then $\hs(\fa^m)= m^n \hs(\fa)$. Applying this property and dividing by $m \cdot \ell$, gives that
\[
\frac{\hs\left(\fa_m(v)\right)^{1/n}}{m} \leq \frac{\hs \left(\Jac_X+\fm_x^{mr} \right)^{1/n}}{m} + \frac{m+E}{m} \cdot
\frac{ \hs\left( \fa_{(m+E)\ell } (v)\right)^{1/n}}{(m+E)\ell}.\] 
After letting $\ell\to \infty$, we obtain
\[
\frac{\hs\left(\fa_m(v)\right)^{1/n}}{m} \leq \frac{\hs\left(\Jac_X+\fm_x^{mr} \right)^{1/n}} {m} +\frac{m+E}{m} \vol(v)^{1/n} .\]

Since $\vol(v)^{1/n} \leq B^{1/n}$, the assertion will follow if we show that
\[
\lim_{m \to \infty} \frac{ \hs\left(\Jac_X+ \fm_x^{mr} \right)^{1/n}}{m} = 0 .
\]
Choose $h \in \Jac_X \cdot \cO_{X,x}$ that is nonzero and set $R\coloneqq \cO_{X,x}/(h)$ and $\tilde{\fm}_x \coloneqq  \fm_x \cdot R$. We have 
\[\lim_{m \to \infty} \frac{ \hs(\Jac_X+ \fm_x^{mr} )^{1/n}}{m} 
= \lim_{m \to \infty}  \frac{ \length ( \cO_{X,x} /(\Jac_X+ \fm_x^{mr}) )^{1/n}}{m} \leq \lim_{m\to \infty} \frac{ \length(R/ \tilde{\fm}_x^{mr})^{1/n}}{m} .\]
The last limit is $0$, since 
\[
\lim_{m \to \infty} \frac{ \length(R/ {\tilde{\fm}_x}^{mr})}{m^{n-1}/(n-1)!} = \hs({\tilde{\fm}_x}^r)  < \infty .\]
\end{proof}

\section{Normalized Volumes}\label{normal}
For this section, we fix $X$ an $n$-dimensional klt variety and $x\in X$ a closed point. 
As introduced in  \cite{Li1}, the \emph{normalized volume}  of  a valuation $v\in \Val_{X,x}$ is defined as
\[
\nvol (v)\coloneqq A_X(v)^n \vol(v) .\]
In the case when $A_X(v)=+\infty$ and $\vol(v)=0$, we set $\nvol(v)\coloneqq +\infty$. The word ``normalized'' refers to the property that $\nvol(\la v) = \nvol(v)$ for $\la\in \RR_{>0}$. 

Given a graded sequence $\ab$ of $\fm_x$-primary ideals on $X$, we define a similar 
invariant. We refer to 
\[
\lct(\ab)^n \hs(\ab)
\]
as the \emph{normalized multiplicity} of $\ab$. Similar to the normalized volume,  when $\lct(\ab)=+\infty$ and $\hs(\ab)=0$, we set $\lct(\ab)^n \hs(\ab)\coloneqq +\infty$. 
The above invariant was looked at in \cite{DFEM} and \cite{mus}.

The following lemma provides elementary information on the normalized multiplicity. The proof is left to the reader. 
\begin{lemma}\label{details}
Let $\fa$ be an $\fm_x$-primary ideal and $\ab$ a graded sequence of $\fm_x$-primary ideals on $X$. 
\begin{enumerate}
\item If $\lct(\ab)^n \hs(\ab)< +\infty$, then 
\[
\lct(\ab)^n \hs(\ab) =  \lim_{m\to \infty}  \lct(\fa_m)^n \hs(\fa_m).
\]
\item If  $\fb_\bullet$ is a graded sequence given by $\fb_m\coloneqq  \fa^m$, then 
\[
\lct(\fa)^n \hs(\fa) = \lct(\fb_\bullet)^n \hs(\fb_\bullet). 
\]
\item  If $\fa_{N\bullet}$ is the graded sequence whose $m$-th term is $\fa_{N\cdot m}$, then 
\[
\lct(\fa_\bullet)^n \hs(\fa_\bullet) 
=
\lct(\fa_{N \bullet})^n \hs(\fa_{N \bullet}) .\]
\end{enumerate}
\end{lemma}
\begin{remark}
Fix $\delta>0$. If $\ab$ a graded sequence of $\fm_x$-primary ideals such that $\fa_m \subseteq \fm_x^{ \lfloor \delta m\rfloor }$ for all $m$, then 
\[
\lct(\ab)^n \hs(\ab) < +\infty.
\]
It is always the case that $\hs(\ab)< +\infty$, since $\hs(\ab) \leq \hs(\fa_1)$. 
 The assumption that $\fa_m \subseteq \fm_x^{ \lfloor \delta m\rfloor }$  gives that $\lct( \ab) \leq \lct(\fm_x)/\delta$, the latter of which is $<+\infty$.
\end{remark}

The following proposition relates the normalized volume, an invariant of valuations, to the normalized multiplicity, an invariant of graded sequences of ideals. 
\begin{proposition}[ {\cite[Theorem 27]{Liu}}] \label{relate}
The following equality holds:
\[
\inf_{v\in \Val_{X,x}} \nvol(v)
= 
\inf_{\ab \, \fm_x \text{-primary}   } \lct(\ab)^n \hs(\ab)
=
\inf_{\fa\, \fm_x \text{-primary}  }\lct(\fa)^n \hs(\fa).
\]
\end{proposition}
The previous statement first appeared in \cite{Liu}. In the case when $x\in X$ is a smooth point, it was partially given in \cite[Example 3.7]{Li1}. We provide Liu's proof, since the argument will be useful to us. The proposition is a consequence of the following lemma.

\begin{lemma}[\cite{Liu}]  \label{tech} The following statements hold.

\begin{enumerate}
\item If $\ab$ is a graded sequence of $\fm_x$-primary ideals and $v\in \Val_{X,x}$ 
computes $\lct(\ab)$ (i.e. 
$
\frac{A(v)}{v(\ab)} = \lct(\ab)$ ), then 
\[\nvol(v) \leq \lct(\ab)^n\hs(\ab) .\]
\item If $v\in \Val_{X,x}$, then 
\[
\lct(\ab(v) )^n \hs(\ab(v) )
\leq 
\nvol(v).\] 
\end{enumerate}
\end{lemma}

\begin{proof}
To prove (1), we first rescale $v$ so that $v(\ab)=1$. 
Thus, $A_X(v) = A_X(v) /v(\ab)  = \lct(\ab)$.
Since 
\[
1 = v(\ab) \coloneqq  \inf_{m\geq 0} \frac{ v(\fa_m) }{m}
,\]
we see
$v(\fa_m) \geq  m$
and, thus, $\fa_m \subseteq \fa_m(v)$ for all $m$. This implies $
\hs(\ab(v)) \leq \hs(\ab)$, and the desired inequality follows. 

In order to show (2), we note  
\[
\lct( \ab(v)) \coloneqq  \min_{w} \frac{ A_X(w)}{w(\ab(v))} \leq \frac{A_X(v)}{v(\ab(v))} = A_X(v),\]
where the last equality follows from Lemma \ref{self}. Thus, 
\[
\lct(\ab(v))^n \hs (\ab(v)) \leq A_X(v)^n \hs(\ab(v)) = \nvol(v).\] 
\end{proof}

\begin{proof}[Proof of Proposition \ref{relate}]
The first equality follows immediately from the previous proposition and the fact that given a graded sequence $\ab$, there exists a valuation $v^*\in \Val_{X}$ that computes $\lct(\ab)$ by Theorem \ref{computingtheorem}. The last equality follows from Lemma \ref{details}.
\end{proof}

\begin{remark}
Above, we provided a dictionary between the normalized volume of a valuation and the normalized multiplicity of a graded sequence of ideals. The normalized multiplicity also extends to a functional on the set of (formal) plurisubharmonic functions in the sense of \cite{BFJ08}. In a slightly different setting, similar functionals, arising from non-Archimedean analogues of functionals in K{\"a}hler geometry, were explored in \cite{BHJ}. 
\end{remark}

\subsection{Normalized volume minimizers}

\begin{proposition}
If there exists a graded sequence of  $\fm_x$-primary ideals $\tilde{\fa}_\bullet$ such that
\[
\lct(\tilde{\fa}_\bullet)^n \hs(\tilde{\fa}_\bullet) = \inf_{\ab\, \fm_x \text{-primary}  }\lct(\ab)^n \hs(\ab),
\]
then there exists $v^\ast \in \Val_{X,x}$ that is a minimizer of $\nvol_{X,x}$.
Furthermore, if there exists an $\fm_x$-primary ideal $\tilde{\fa}$ such that 
\[
\lct(\tilde{\fa})^n \hs(\tilde{\fa}) = \inf_{\fa \, \fm_x \text{-primary}  }\lct(\fa)^n \hs(\fa)
,\]
then we may choose $v^\ast$ to be divisorial. \end{proposition}

\begin{proof}
Assume there exists such a graded sequence $\tilde{\fa}_\bullet$. By Theorem \ref{computingtheorem}, we may choose a valuation $v^\ast$ that computes $\lct(\tilde{\fa}_\bullet)$. By Lemma \ref{tech}, 
\[
\nvol(v^\ast) \leq \lct(\tilde{\fa}_\bullet)^n \hs(\tilde{\fa}_\bullet) .\]
By Proposition \ref{relate}, the result follows. 

When there exits such an ideal $\tilde{\fa}$, the same argument shows that if $v^\ast$ computes $\lct(\tilde{\fa})$, then $v^\ast$ is our desired valuation. Furthermore, we may choose $v^\ast$ divisorial.  
\end{proof}

\begin{lemma}\label{lcttech}
 If $v^\ast$ is a minimizer of $\nvol_{X,x}$, then 
\[
A_X(v^\ast) \leq \frac{ A_X(w) }{ w( \ab(v^\ast))}
\]
for all $v\in \Val_{X,x}$. 
Furthermore, equality holds if and only if $w= \la v^\ast $ for some $\la \in \RR_{>0}$.
\end{lemma}

\begin{remark}
The above technical statement can be restated as follows. If $v^\ast$ is a normalized volume minimizer, then $v^\ast$ computes $\lct( \ab(v^\ast))$ and $v^\ast$ is the only valuation (up to scaling) that computes $\lct( \ab(v^\ast))$. 

A conjecture of Jonsson and Musta{\c{t}}{\u{a}} states that valuations computing log canonical thresholds of graded sequences on smooth varieties are always quasimonomial \cite[Conjecture B]{JM}. Their conjecture  in the klt case implies \cite[Conjecture 6.1.3]{Li1}, which says that normalized volume minimizers are quasimonomial.
\end{remark}

\begin{proof}
We fix $w\in \Val_{X,x}$ and rescale $w$ so that $w(\ab(v^\ast))=1$. Thus, we are reduced to showing that $A_X(v^\ast) \leq A_X(w)$ and equality holds if and only if $w=v^\ast$.

Definitionally, we have
\[
1=w(\ab(v^\ast)) \coloneqq  \inf_{m\geq 0} \frac{ w(\fa_m(v^\ast))}{m}, \]
and, thus, $w(\fa_m(v^\ast)) \geq m$. The latter implies that $\fa_m(v^\ast) \subseteq \fa_m(w)$, so
\[
\vol(w) \leq \vol(v^\ast).\]

If $A_X(w) < A_X(v^\ast)$, then \[
A_X(w)^n \vol(w) < A_X(v^\ast)^n \vol(v^\ast)\]
and this would contradict our assumption on $v^\ast$. 
Furthermore, if $A(v^\ast) = A(w)$, then we must have that $\vol(v^\ast) = \vol(w)$. Since $v^\ast \leq w$ and $\vol(v^\ast) = \vol(w)$, then $v^\ast = w$ by \cite[Proposition 2.12]{LiXu}.
\end{proof}

\begin{proposition}\label{rational}
Let  $v^\ast \in \Val_{X,x}$  be a minimizer of $\nvol_{X,x}$. If  $v^\ast=\ord_E$, where $E$ is a prime divisor on a normal variety which is proper and birational over $X$, then 
\begin{enumerate}
\item the graded $\cO_X$-algebra $ \cO_X \oplus  \fa_1(v^\ast) \oplus \fa_2(v^\ast) \oplus \cdots $ is finitely generated,
\item the valuation $v^\ast$ corresponds to a Kollar component (See \cite{LiXu}),
\item the number $\nvol(v^\ast)$ is rational.
\end{enumerate}
\end{proposition}

The previous proposition was independently observed in \cite[Theorem 1.5]{LiXu}. In fact, prior to our contribution, the original draft of \cite{LiXu} proved that if (1) holds then (2) holds. Our argument is different from that of \cite{LiXu}.

\begin{proof} 
By Lemma \ref{lcttech}, it follows that $\lct(\ab(v^\ast)) = A(v^\ast)$.
The finite generation of the desired $\cO_X$-algebra is a consequence of \cite[Theorem 1.4.1]{Blum}. Additionally, the second sentence of Lemma \ref{lcttech} allows us to apply \cite[Proposition 4.4]{Blum}. Thus, $v^\ast$ corresponds to a Kollar component.

To show that $\nvol(v^\ast)$ is rational, we note that the finite generation statement of (1) implies there exists $N>0$ so that
$\fa_{mN} (v^\ast)=( \fa_N(v^\ast))^m$ for all $m\in \NN$ \cite[Lemma II.2.1.6.v]{EGA}. By Lemma \ref{tech}, 
\[
\lct(\ab(v^\ast))^n \hs(\ab(v^\ast))\leq \nvol(v^\ast).\]
 Lemma \ref{details} implies  
 \[
 \lct(\ab(v^\ast))^n \hs(\ab(v^\ast)) = \lct( \fa_N)^n \hs(\fa_N),\] and the latter is a rational number.
\end{proof}

\section{Limit Points of Collections of Graded Sequences of Ideals}\label{seclimits}
In this section we construct a space that parameterizes graded sequences of ideals on a fixed variety $X$. We  use this parameter space to find ``limit points'' of a collection of graded sequences of ideals on $X$. The ideas behind this construction arise from the work of de Fernex-Musta{\c{t}}{\u{a}}  \cite{FM},  Koll{\'a}r \cite{Which}, and de Fernex-Ein-Musta{\c{t}}{\u{a}} \cite{Smooth} \cite{Bound}.

Before explaining our construction, we set our notation. We fix an affine variety $X=\Spec A$, where $A= k[x_1,\ldots,x_r] / \fp$. Let $\varphi$ denote the map
\[
R= k[x_1,\ldots, x_r] \overset{\varphi}{\longrightarrow} A= k[x_1,\ldots,x_r] / \fp
\, .\]
 We set $\fm_R\coloneqq  (x_1,\ldots,x_r)$ and assume that  $\fp \subset \fm_R$. Thus,  $\fm_A= \varphi(\fm_R)$ is a maximal ideal of $A$. 
\subsection{Parameterizing Ideals}

We fix an integer $d>0$ and seek to parameterize   ideals $\fa \subset A$ containing $\fm_A^d$ and contained in $\fm_A$. Since 
$\fm_R^d \subseteq \varphi^{-1}( \fm_A^d)$, such an ideal $\fa\subseteq A$ can be generated by $\fm_A^d$ and  images of polynomials from $R$ of deg $<d$.  Since there are $n_d= \binom{r+d-1}{r}-1$ monomials of positive degree $< d$ in $R$, any such  ideal $\fm_A^d \subseteq \fa \subseteq \fm_A$ can be generated by $\fm_A^d$ and the image of $n_d$ linear combinations of monomials. After setting $N_d = n_d^2$, we get a map 
\[
\left\{ \text{ $k$-valued points of }
\Af^{N_d} \right\} \longrightarrow  
\left\{ \text{ ideals $\fa \subseteq A$ s.t.  $\fm_A^d \subseteq \fa \subseteq \fm_A $}  \right\},
\]
where $\Af^{N_d}$ parameterizes coefficients and generators of such ideals.
The above map is surjective, but not injective (generators of an ideal are not unique). Additionally, we have a universal ideal $\sA \subset \cO_{X\times \Af^{N_d}}$ such that $\sA$ restricted to the fibers of $p:X\times \Af^{N_d}\to \Af^{N_d}$ give us our $\fm_A$-primary ideals. 

The construction in the previous paragraph is not original. The paragraph follows the exposition of \cite[Section 3]{FM}.

\subsection{Parameterizing Graded Sequences of Ideals}\label{param} We proceed to parameterize graded sequences of ideals $\ab$ of $A$ satisfying 
\[
(\dagger) \, \, \text{$\fm_A^m \subseteq \fa_m \subseteq \fm_A $ for all $m\in \NN$.}\] 
We set
\[
H_d  \coloneqq  \Af^{N_1} \times \cdots \times \Af^{N_d},  \]
where $N_i$ is chosen as in the previous section. For $d>c$, let $
\pi_{d,c} : H_d \to H_c$
denote the natural projection maps.
Our desired object is the following projective limit
\[
H =  \varprojlim H_d.
\]
For $d>0$, let $\pi_d : H \to H_d$ denote the natural map. 

Note that the above projective limit exists in the category of schemes, since the maps in our directed system are all affine morphisms. Indeed, $H$ is isomorphic to an infinite-dimensional affine space. 

Since a $k$-valued point of $H$ is simply a sequence of $k$-valued points of $\Af^{N_d}$ for all $d\in \NN$, we have a surjection
\[
\left\{ \text{$k$-valued points of } H \right\} \longrightarrow  
\left\{ \text{sequences of ideals $\fb_\bullet$ of $A$ satisfying $(\dagger)$}  \right\}
.\]
Note that the sequences of ideals on the right hand side are not necessarily graded.

Given a sequence of ideals $\fb_\bullet$, we can construct a graded sequence $\fa_\bullet$ inductively by setting $\fa_1 \coloneqq  \fb_1$ and 
\[
\fa_q \coloneqq  \fb_q + \sum_{m+n=q} \fa_m \cdot \fa_n.\] If $\fb_\bullet$ was graded to begin with, then $\ab= \fb_\bullet$.
By the construction, it is clear that $\fa_m \cdot \fa_n \subseteq \fa_{m+n}$. 
Thus, we have our desired map 
\[
\left\{ \text{$k$-valued points of } H \right\} \longrightarrow  
\left\{ \text{graded sequences of ideals $\ab$  of $A$ satisfying $(\dagger)$}  \right\}
.\]
Additionally, we have a universal graded sequence of ideals $\sA_\bullet = \{ \sA_{m}\}_{m\in \NN}$ on $X\times H$. We will often abuse notation and refer to similarly defined ideals $\sA_1,\ldots, \sA_d$ on $X\times H_d$. 

The following technical lemma will be useful in the next proposition. The proof of the lemma relies on the fact that every descending sequence of non-empty constructible subsets of a variety over an uncountable field has nonempty intersection. 
\begin{lemma}\label{exist}
If $\{W_d \}_{d\in \NN}$ is a collection of nonempty subsets of $H_d$ such that 
\begin{enumerate}
\item $ W_d \subset H_d$ is a constructible, and 
\item 
\[
W_{d+1} \subseteq 
\pi_{d+1,d}^{-1} (W_d)
\]
for each $d\in \ZZ_{>0}$,
\end{enumerate}
 then there exists a $k$-valued point  in 
 \[
 \bigcap
 _{d\in \NN}  \pi_{d}^{-1} \left( W_d \right) \] 
\end{lemma}

\begin{proof}
Note that a $k$-valued point in the above intersection  is equivalent to  a sequence of closed points
$\{x_d \in W_d\}_{d\in \NN}$ such that $\pi_{d+1,d}(x_{d+1})=x_d$. We proceed to construct such a sequence. 

We first look to find a candidate for $x_1$. Assumption (2) implies 
\[
W_1 \supseteq  \pi_{2,1}(W_2) \supseteq \pi_{3,1}(W_3) \supseteq \cdots 
\]
is a descending sequence of nonempty sets.
Note that $W_1$ is constructible by assumption and so are $\pi_{d,1}(W_d)$ for all $d$ by Chevalley's Theorem \cite[Exercise II.3.9]{Har}. Thus, 
\[
 W_1 \cap \pi_{2,1}(W_2) \cap \pi_{3,1}(W_3) \cap \cdots
 \]
 is non-empty and we may choose a point $x_1$ lying in the set. 
 
 Next, we look at 
 \[
 W_2 \cap  \pi_{2,1}^{-1}(x_1)  \supseteq   \pi_{3,2}(W_3)  \cap \pi_{2,1}^{-1}(x_1)  \supseteq \pi_{4,2}(W_4) \cap \pi_{2,1}^{-1}(x_1),
\]
and note that for $d\geq 2$ each $ \pi_{d,2}(W_d)  \cap \pi_{2,1}^{-1}(x_1)$  is nonempty by our choice of $x_1$. By the same argument as before, we see
 \[
\pi_{2,1}^{-1}(x_1)   \cap  W_2 \cap  \pi_{3,2}(W_3)  \cap \pi_{4,2}(W_4)  \cap \cdots
\]
is non-empty and contains a closed point $x_2$. 
Continuing in this manner, we  construct the desired sequence.  
\end{proof}

\subsection{Finding Limit Points}

The proof of the following proposition relies on the previous construction of a space that parameterizes graded sequences of ideals. The proof is inspired by arguments in \cite{Which} and \cite{Smooth} \cite{Bound}.

\begin{proposition}\label{limits}
Let $X$ be a klt variety and $x\in X$ a closed point. Assume there exists a collection of graded sequences of $\fm_x$-primary ideals  $\{\ab^{(i)}\}_{i\in \NN}$   and $\la\in \RR$ such that the following hold:
\begin{enumerate}
\item (Convergence from Above) For every $\epsilon >0$, there exists positive constants $M,N$ so that 
\[
\lct( \fa_{m}^{(i)}) ^n \hs( \fa_m^{(i)}) \leq \la + \epsilon 
\]
for all $m\geq M$ and $i\geq N$.
\item (Boundedness from Below)  For each $m,i \in \NN$, we have  
\[
\fm_x^m \subseteq \fa_m^{(i)}.\]
\item (Boundedness from Above) There exists $\delta >0$ such that 
\[
\fa_m^{(i)} \subseteq \fm^{ \lfloor m \delta \rfloor}\]
for all $m,i\in \NN$. 
\end{enumerate}
Then, there exists a graded sequence of $\fm_x$-primary ideals $\tilde{\fa}_\bullet$ on $X$ such that 
\[
\lct( \tilde{\fa}_\bullet)^n \hs(\tilde{\fa}_\bullet) \leq \la .\]
\end{proposition}

\begin{proof}
It is sufficient to consider the case when $X$ is affine. Thus, we may assume that $X= \Spec A$ and $A= k[x_1,\ldots,x_r]/\fp$ as in the beginning of this section. Additionally we may assume that $x\in X$ corresponds to the maximal ideal $\fm_A$. 
We recall that Section \ref{param} constructs a variety $H$ parameterizing graded sequences of ideals on $X$ satisfying ($\dagger$). Additionally, we have finite dimensional truncations $H_d$ that parameterize the first $d$ elements of such a sequence.

Each graded sequence $\ab^{(i)}$ satisfies $(\dagger)$ by assumption (2) and (3). Thus, we may choose a point $p_i \in H$ corresponding to $\ab^{(i)}$. Note that $\pi_d(p_i)\in H_d$ corresponds to the first $d$-terms of $\ab^{(i)}$. 
\\

\noindent {\bf Claim 1:} We may choose infinite subsets $\NN \supset I_1 \supset I_2 \supset \cdots $
and set 
\[
Z_d \coloneqq  \overline{ \{ \pi_d(p_i)  \vert i \in I_d \}} \]
such that ($\ast \ast$) holds. 
\begin{center} 
($\ast \ast$)
If $Y\subsetneq Z_d$ is a closed set, there are only finitely many $i \in I_d$ such that $\pi_d(p_i) \in Y$. 
\end{center} 

To prove Claim 1, we construct such a sequence inductively. First, we set $I_1 = \NN$. Since $H_1\simeq \Af^0$ is a point, ($\ast \ast$) is trivially satisfied for $d=1$. 
After having chosen $I_d$, choose $I_{d+1} \subset I_d$ so that ($\ast \ast$) is satisfied for $Z_{d+1}$.  By the Noethereanity of $H_d$, such a choice is possible. \\

\noindent {\bf Claim 2:} 
We have the inclusion $Z_{d+1} \subseteq \pi_{d+1,d}^{-1}(Z_d)$ for all $d\geq 1$.

The proof of Claim 2 follows from the definition of $Z_d$. Since $\pi_d (p_i) \in Z_d$ for all $i\in I_d$ and $I_d \supseteq I_{d+1}$, it follows that $\pi_{d+1,d}^{-1}(Z_d)$ is a closed set containing $\pi_{d+1}(p_i)$ for $i \in I_{d+1}$. The closure of the latter set of points is precisely $Z_{d+1}$. \\

\noindent {\bf Claim 3}
If $p \in Z_d$ is a closed point, we have $\sA_d\vert_p \subseteq \fm^{ \lfloor d \delta \rfloor}$. 

We now prove Claim 3. The set
$\{p \in H_d  \,\vert \, \cA_d\vert_p \subseteq \fm^{ \lfloor d \delta \rfloor}\}$
is a closed in $H_d$. By assumption 3,  $\pi_d(p_i)$ lies in the above closed set for all $i \in I_d$.
Thus, $Z_d \subseteq \{p \in H_d  \,\vert \, \cA_d\vert_p \subseteq \fm^{ \lfloor d \delta \rfloor}\}$ \\

We now return to the proof of the proposition.We look at the normalized multiplicity of the ideals parameterized by $Z_d$.
By Propositions \ref{cmult} and \ref{clct}, for each $d$, we may choose a nonempty open set  $U_d\subseteq Z_d$ such that 
\[
\lct(\sA_{d}\vert_p)^n
\hs( \sA_{d}\vert_p)= \la_d
\]
is constant for $p \in U_d$. 
Set 
\[
I_d^\circ = \{ i \in I_d \, \vert \, \pi_{d}(p_i) \in U_d \} \subseteq I_d,\]
 and note that $I_d \setminus I_d^\circ$ is finite. If this was not the case, then ($\ast \ast$) would not  hold.
 
The finiteness of $I_d \setminus I_d^\circ$ has two consequences. First, 
\[
\lim_{d\to \infty} \sup \la_d \leq \la,
\]
since $\pi_d(p_i) \in U_d$ for all $i\in I_d^\circ$ and assumption (1) of this proposition. 
Second, since $I_{d+1} \subset I_d$ for $d\in \NN$, we have 
\[
I_d^\circ \cap I_{d-1}^\circ \cdots \cap I_1^\circ\neq \emptyset.\]

\noindent {\bf Claim 4}: There exists  a $k$-valued point $\tilde{p} \in H$ such that that $\pi_d(\tilde{p}) \in U_d$ for all $d\in \ZZ_{>0}$.  

Proving this claim will complete the proof. Indeed, a point $\tilde{p}\in H$ corresponds to a graded sequence of $\fm_x$-primary ideals $\tilde{\fa}_\bullet$.
Since $\pi_d(\tilde{p})\in U_d$, we will have $
\lct(\tilde{\fa}_d)^n \hs(\tilde{\fa}_d) = \la_d$. Additionally, Claim 2 implies  and
$
\tilde{\fa}_d \subset \fm_x^{ \lfloor d \delta \rfloor }$
for all $d\in \ZZ_{>0}$. 
Thus,
\[
\lct(\tilde{\fa}_\bullet)^n \hs(\tilde{\fa}_\bullet) = \lim_{d \to \infty} \lct ( \tilde{\fa}_d ) \hs(\tilde{\fa}_d) \leq \lim_{d \to \infty} \sup \la_d \leq \la, 
\]
and the proof will be complete. \\

We are left to prove Claim 4. In order to do so, we will apply Lemma \ref{exist} to find such a point $\tilde{p}\in H$. First, we define constructible sets $W_d\subseteq H_d$ inductively. Set $W_1=U_1$. After having chosen $W_d$, set $W_{d+1} = \pi^{-1}_{d+1,d}(W_d) \cap U_d$. For each $d\in \NN$  we have the following :
 \begin{itemize}
 \item $W_d$ is open in $Z_d$ and, thus, constructible in $H_d$. 
 \item  $W_d$ is nonempty, since $W_d$ contains $\pi_d(p_i)$ for all $i \in I_d^\circ \cap I_{d-1}^\circ \cap \cdots \cap I_1^\circ$, which is nonempty. 
 \end{itemize}
 By Lemma \ref{exist}, there exists a $k$-valued point $\tilde{p}\in H$ such that $\pi_d(\tilde{p})\in W_d \subset U_d$ for all 
$d\in \ZZ_{>0}$. This completes the proof.
\end{proof}

\begin{remark}
In the previous proof, we construct a graded sequence of ideal $\tilde{\fa}_\bullet$ based on a collection of graded sequences $\{\ab^{(i)}\}_{i \in \NN}$. While the construction of $\tilde{\fa}_\bullet$ is inspired by past constructions of generic limits, $\tilde{\fa}_\bullet$ is not a generic limit in the sense of \cite[28]{Which}.

We construct the precise analog as follows.
We set
\[
Z: = \bigcap_{d} \pi^{-1}_d (Z_d) \subseteq H,\]
with the $Z_d$'s defined in the previous proof.
The generic point of $Z$ gives a map $\Spec(K(Z)) \to H$, where $K(Z)$ is the function field of $Z$. Thus, we get a graded sequence of ideals $\widehat{\fa}_\bullet$ on $X_{K(Z)}$, the base change of $X$ by $K(Z)$. 

In the previous proof, we wanted to construct a graded sequence on $X$, not a base change of $X$. 
Thus, $\tilde{\fa}_\bullet$ was chosen to be a graded sequence  corresponding to a very general point in $Z$. 
\end{remark} 

\section{Proof of the Main Theorem}\label{secmain}
In this section we prove  the {\hyperref[a]{Main Theorem}}. To prove the theorem, we apply the construction from Section \ref{seclimits}.

\begin{proof}[Proof of the {\hyperref[a]{Main Theorem}}]

We fix a klt variety $X$ and a closed point $x\in X$.
Next, we choose a sequence of valuations $\{v_i\}_{i\in \NN}$ in $\Val_{X,x}$ such that 
\[
 \lim_{i} \nvol(v_i) = \inf_{v \in \Val_{X,x}} \nvol(v).\] Additionally, after scaling our valuations, we may assume that $v_i(\fm_x)=1$ for all $i\in \NN$. Note that this implies that $
 \fm_x^m \subset \fa_m(v_i)
 $
 for all $m$. 
 
  We claim that $\{ \ab(v_i) \}_{i \in \NN}$ satisfy the hypotheses of Proposition \ref{limits} with $\la = \inf_{v\in \Val_{X,x}} \nvol(v)$. 
 After showing that this is the case, we will have that there exists a graded sequence of $\fm_x$-primary ideals  $\tilde{\fa}_\bullet$ such that 
\[
\lct(\tilde{\fa}_\bullet)^n \hs(\tilde{\fa}_\bullet) \leq \inf_{v\in \Val_{X,x} } \nvol(v).\]
By Theorem \ref{computingtheorem}, there exists a valuation $v^\ast\in \Val_{X,x}$ that computes $\lct(\tilde{\fa}_\bullet)$. Thus, 
\[
\nvol(v^\ast)\leq \lct(\tilde{\fa}_\bullet)^n \hs(\tilde{\fa}_\bullet) = \inf_{v\in \Val_{X,x}} \nvol(v),\]
where the first inequality follows from Lemma \ref{tech}. Thus, $v^\ast$ will be our normalized volume minimizer.

It is left to show that $\{ \ab(v_i)\}_{i \in \NN}$ satisfies the hypotheses of Proposition \ref{limits}. Hypothesis (1) follows from Proposition \ref{boundbelow},  (2) from the assumption that $v_i(\fm)=1$ for all $i\in \NN$, and (3)  from Proposition \ref{convabove}.
\end{proof}

We proceed to prove the two propositions mentioned in the previous paragraph. We emphasize that estimates from  \cite{Li1} are essential in the proof of the following lemma and propositions.

 \begin{lemma}\label{bound}
 With the notation above, 
there exist positive constants $E,B$ such that (1) $A_X(v_i) \leq E$ and (2) $\vol(v_i) \leq B$ for all $i \in \NN$.  
 \end{lemma}
 
 \begin{proof}
 By \cite[Theorem 3.9]{Li1}, there exists a constant $C$ such that 
 \[
 A_X(v) \leq C \cdot v(\fm) \nvol(v)\]
 for all $v\in \Val_{X,x}$. Thus, we may set $E\coloneqq  C \cdot \sup_{i} \nvol(v_i) < +\infty$

 The bound on the volume follows from the inclusion $\fm_x^m \subset \fa_m(v_i)$ for all $m\in \NN$. The inclusion gives that 
 \[
 \vol(v_i) = \lim_{m \to \infty} \frac{ \hs (\fa_m(v_i))}{m^n} \leq   \lim_{m \to \infty} \frac{ \hs (\fm_x^n)}{m^n}= \hs(\fm_x). \]
 \end{proof}
 
 \begin{proposition} \label{convabove} With the notation above, there exists $\delta>0$ such that 
 \[
 \fa_m(v_i) \subseteq \fm_x^{\lfloor \delta m \rfloor}\]
 for all $m,i\in \NN$. 
 
 \end{proposition}
 
 \begin{remark}
 Note that for an ideal $\fa$, the order of vanishing of $\fa$ along $x$ is defined to be
 \[
 \ord_x(\fa) \coloneqq  \max \{ n \,\vert\, \fa \subseteq \fm_x^n \}.\]
 To prove the above proposition, it is sufficient to find $\delta'>0$ such that
 \[
\delta' m \leq  \ord_x(\fa_m(v_i))  \]
 for all $m,i \in \NN$. 
 \end{remark}

 \begin{proof}
 
 By \cite[Proposition 2.3]{Li1}, there exists a constant $C$ such that for all $v\in \Val_{X,x}$ and $f\in \cO_{X,x}$, 
 \[
 v(f)\leq C\cdot  A_X(v) \ord_x(f).\]
 Thus, 
 \[
  m \leq v_i(\fa_m(v_i)) \leq C \cdot A_X(v_i) \ord_x( \fa_m(v_i)),\]
  and 
  \[
  \frac{m}{C A_X(v_i)} \leq \ord_x(\fa_m(v_i))
  .\]
  By Lemma \ref{bound}, there exists a positive constant $E$ such that $A_X(v_i) \leq E$ for all $i$. Thus, 
    \[
  \frac{m}{C \cdot E } \leq \ord_x(\fa_m(v_i))
  .\]
 \end{proof}

 \begin{proposition} \label{boundbelow}
 With the notation above, 
 for $\epsilon >0$, there exist positive constants $M,N$ such that 
 \[
 \lct( \fa_m(v_i))^n \hs(\fa_m(v_i)) \leq \inf_{v\in \Val_{X,x}} \nvol(v) + \epsilon .\]
 for all $m\geq M$ and $i \geq N$. 
 \end{proposition}
 \begin{proof}
Since $\nvol(v_i)$ converges to $\inf_{v\in \Val_{X,x}} \nvol(v)$ as $i \to \infty$, we may choose $N$ so that
\[
\nvol(v_i) \leq \inf_{v\in \Val_{X,x}}\nvol(v) + \epsilon/2
\]
for all $i\geq N$. By Lemma \ref{bound}, we have $E: = \sup A_X(v_i) < \infty$. Additionally, Lemma \ref{bound} allows us to  apply Proposition \ref{volumebound} to find  a constant $M$ so that 
\[
\hs(\fa_m(v_i)) \leq \vol(v_i) + \epsilon/(2 E^n)
\]
for all integers $m\geq M$. Thus,  
\[
\lct(\fa_m(v_i))^n \hs(\fa_m(v_i)  )\leq A_X(v_i)^n \left( \vol(v_i) + \epsilon/(2 E^n) \right) \leq \nvol(v_i) +\epsilon/2 .\]
 for all $m\geq M$ and $i\in \NN$. We conclude that 
 \[
\lct(\fa_m(v_i)) \hs(\fa_m(v_i))  \leq \nvol(v_i) +\epsilon/2 \leq \inf_{v\in \Val_{X,x} }\nvol(v) + \epsilon \]
for all $m\geq M$ and $i\geq N$.
\end{proof}

\section{The Normalized Volume over a Log Pair}\label{seclog}

The normalized volume function has been studied in the setting of log pairs \cite{LiXu} \cite{LiLiu}. We explain that the arguments in this paper extend to the setting where $(X,\Delta)$ is a klt pair. 

\subsection{Log Pairs}
We say
$(X, \Delta)$ is a \emph{log pair} if $X$ is a normal variety, $\Delta$ is an effective $\QQ$-divisor on $X$, and $K_X+\Delta$ is $\QQ$-Cartier. 

\subsection{Log Discrepancies}
If $(X,\Delta)$ is a log pair,  the \emph{log discrepancy} function $A_{(X,\Delta)} : \Val_{X} \to \RR \cup \{+\infty\}$ is defined as follows.  Let $f:Y\to X$ be a log resolution of $(X, \Delta)$. Choose $\Delta_Y$ so that $K_Y + \Delta_Y = f^\ast (K_X+\Delta)$ where $K_Y$ and $K_X$ are chosen so that $f_\ast K_Y = K_X$. For $v\in \Val_X$, we set 
\[
A_{(X,\Delta)}(v)\coloneqq  A_Y(v)+ v(\Delta_Y).\]

Alternatively, we could define $A_{(X,\Delta)}$ to be the unique lower semicontinous function on $\Val_X$ that respects scaling and satisfies the following property.  If $ E\subset Z\overset{g}{\to} X$ is a proper birational morphism, $Z$ a normal variety, and $E$ a prime divisor on $Z$, then \[ A_{(X,\Delta)}(\ord_E)= 1 +\ord_E( K_Z-g^\ast(K_X+\Delta)).\]

If $(X,\Delta)$ is a log pair, we say $(X,\Delta)$ is a \emph{klt pair} if $A_X(v)> 0$ for all $v\in \Val_X$. It is sufficient to check this condition on a log resolution of $(X,\Delta)$. 

\subsection{Normalized Volume Minimizers}
When $(X,\Delta)$ is klt pair and $x\in X$ is a closed point, the normalized volume of a valuation $v\in \Val_X$ over the pair $(X,\Delta)$ is defined to be
\[
\nvol_{(X, \Delta), x}(v) \coloneqq  A_{(X,\Delta)}(v)^n  \vol(v) . 
\]
We claim that if $(X,\Delta)$ is a klt pair and $x\in X$ is a closed point, then there exists a minimizer of $\nvol_{(X,\Delta), x}$. 

The main subtlety in extending our arguments to the log setting is in extending Theorem \ref{approx}, which is a consequence of the subadditivity theorem. Takagi proved the following subadditivity theorem for log pairs. 

\begin{proposition}\cite{Takpairs}
Let $(X,\Delta)$ be a log pair,  $\fa, \fb$ ideals on $X$, and  $s,t\in \QQ_{\geq 0}$. For $r\in \ZZ_{>0}$ so that $r(K_X+\Delta)$, we have
\[
\Jac_X \cdot \cJ((X,\Delta), \fa^s \fb^s \cO_X(-r\Delta)^{1/r} ) \subseteq \cJ((X, \Delta), \fa^s) \cJ((X,\Delta), \fb^t).
\]
\end{proposition}

 Takagi's result implies the following generalization of Theorem \ref{approx} for log pairs. The remaining arguments in the paper extend to this setting.
\begin{theorem}
If $(X,\Delta)$ is a klt pair and $v\in \Val_X$ satisfies $A_X(v) < +\infty$, then  
\[
(\Jac_X \cO_X(-r\Delta))^\ell  \cdot \fa_{m}^\ell \subseteq (\Jac_X \cO_X(-r\Delta))^\ell   \cdot \fa_{m\ell} \subseteq \fa_{m-e}^\ell
\]
for every $m\geq e$,
where $\ab : = \ab(v)$ and $e\coloneqq  \lceil A_{(X,\Delta)}(v) \rceil$.
\end{theorem}

\section{The Toric Setting}\label{sectoric}

   We use the notation of \cite{Fulton} for toric varieties. Let $N$ be a free abelian group of rank $n\geq 1$ and  $M=N^*$ its dual. We write $N_\RR\coloneqq  N \otimes \RR$ and $M_\RR\coloneqq  M \otimes \RR$. There is a canonical pairing
   \[
\langle \, \, , \,\, \rangle : N_\RR \times M_\RR \to \RR. 
\]  
We say that an element $u\in N$ is primitive if $u$ cannot be written as $u= a u'$ for $a\in \ZZ_{>1}$ and $u\in N$. 

Fix a maximal dimension, strongly convex, rational, polyhedral cone $\sigma \subset N_\RR$. From the cone $\sigma$, we get a toric variety $X_\sigma = \Spec R_\sigma$, where $R_\sigma = k [ \sigma^\vee \cap M]$.  Let $x\in X_\sigma$ denote the unique torus invariant point of $X_\sigma$.
We write $u_1,\ldots, u_r\in N$ for the primitive lattice points of $N$ that generate the 1-dimensional faces of $\sigma$. Each $u_i$ corresponds to a toric invariant divisor  $D_i$ on $X_\sigma$. Since the canonical divisor is given by $K_{X_\sigma}= - \sum D_i$, the variety $X_\sigma$ is $\QQ$-Gorenstein if and only if there exists $w\in M \otimes \QQ \subset M_\RR$ such that $\langle u_i, w\rangle = 1$ for $i=1,\ldots,r$. 

Given $ u\in \sigma$, we get a \emph{toric valuation} $v_ u\in \Val_X$ defined by
\[v_u \left(  \sum_{m\in M \cap \sigma^\vee} \alpha_v \chi^v \right) =  \min \{ \langle u, v\rangle \, \vert \, \alpha_v \neq 0 \}.\]

If $u \in \sigma^\vee \cap N$ is primitive,  the valuation $v_u$ corresponds to vanishing along a prime divisor on a toric variety proper and birational over $X_\sigma$.
For $u\in \sigma$,    $v_u$ has center equal to $x$ if and only if $u \in \Int(\sigma)$.

Let $\Val_{X_\sigma,x}^{\toric} \subset \Val_{X_\sigma,x}$ denote the valuations on $X_\sigma $ of the form $v_u$ for $u\in \Int(\sigma)$. We refer to these valuations as the \emph{toric valuations} at $x$.
It is straightforward to compute the normalized volume of such a valuation.
 Assume $X_\sigma$ is $\QQ$-Gorenstein and $w$ is the unique vector such that $\langle u_i,w \rangle =1$ for $i=1,\ldots,s$. For $u\in \sigma$, we have 
\[
A_{X_\sigma}(v_u) = \langle u, w \rangle. \]
For $u\in \sigma$ and $a\in \NN$, we set $H_u(m) = \{ v \in M_\RR \, \vert \, \langle u, v \rangle \geq m \}$.  Note that 
\[
\fa_m(v_u) = \left( \chi^v \, \vert \,  v \in H_u(m) \cap \sigma^\vee \cap M \right) .\]
In the case when $u\in \Int (\sigma)$,
\[
\vol(v_u) = n!  \cdot \Vol( \sigma^\vee \cap H_u(1) ),\]
where $\Vol$ denotes the Euclidean volume. 

\subsection{Deformation to the Initial Ideal} 
As explained in \cite{Bud}, when  $X_\sigma \simeq \Af^n$ and $I \subset R_\sigma$, there exists a deformation of $I$ to a monomial ideal. We show that a similar argument extends to our setting. 

We seek to put a $\ZZ_{\geq 0}^n$  order on the monomials of $R_\sigma$. The content of this paragraph is modeled on \cite[Section 6]{Kaveh}.
Fix $y_1,\ldots,y_n\in N \cap \sigma$ that are linearly independent in $M_\RR$. 
Thus, we get an injective map $\rho :M \to \ZZ^n$ by sending 
\[
 v \longmapsto \left( \langle y_1,v \rangle , \ldots,  \langle y_n,v \rangle \right) .\]
Since each $y_i \in \sigma$, we have $\rho( M \cap \sigma^\vee) \subseteq \ZZ_{\geq 0}^n$. 
 After putting the lexigraphic order on $\ZZ_{\geq 0}^n$, we get an order $>$ on the monomials of $R_\sigma$. 
 
An element $f\in R_\sigma$ may be written as a sum of scalar multiples of distinct monomials. The \emph{initial term} of $f$, denoted $\init_>f$, is the greatest term of $f$ with respect to the order $>$. For an ideal $I \subset R_\sigma$, the \emph{initial ideal} of $I$ is 
\[
\init_> I = ( \init_>f \, \vert \, f\in I).\]
The initial ideal satisfies the following property. 

\begin{lemma}\label{deformlength}
If $I\subset R_\sigma$ is an $\fm_x$-primary ideal, then $\init_> I$ is $\fm_x$ primary and \[
\length (R_\sigma / I ) 
= \length(R_\sigma / \init_{>} I) .\]
\end{lemma}

\begin{proof}
The proof is nearly identical to the proof of  \cite[Theorem 15.3]{Bud}.
\end{proof}

Similar to the argument in \cite{Bud}, we construct a deformation of $I$ to $\init_>I$. 
Since $R_\sigma$ is Noetherian, we may choose elements $g_1,\ldots, g_s \in I$ such that 
\[
I = (g_1,\ldots, g_s)\,\,\, \text{ and } \init_>I = (\init_> g_1, \ldots, \init_> g_s).\]
Fix an integral weight $\la : M \cap \sigma^\vee \to  \ZZ_{\geq 0}$
such that 
\[
\init_{>_\la}(g_i) = \init_{>}(g_i)
\]
for all $i$. Note that $>_\la$ denotes the order on the monomials induced by the weight function $\la$. 

Let $R_\sigma [t]$ denote the polynomial ring in one variable over $R_\sigma$.
For $g= \sum \alpha_m \chi^m$, we write $b \coloneqq  \max \{ \la(m)\, \vert \, \alpha_m \neq 0\}$ and set
\[
\tilde{g}: = t^b \sum \alpha_m t^{-\la(m)} \chi^m .\]
 Next, let 
 \[
 \tilde{I} = ( \tilde{g}_1 \,\ldots \tilde{g}_s) \subset R_\sigma [t].\]
 
 For $c\in k$, we write $I_c$ for the image of $\tilde{I}$ under the map $R_\sigma[t]\to R_\sigma$ defined by $t\mapsto 0$. 
 It is clear that $I_1= I$ and $I_0 = \init_{>}I$.

\begin{proposition}\label{deformlct}
If $I$ is an $\fm_x$ primary ideal on $X_\sigma$, then 
$\lct( \init_<(I))  \leq \lct(I)$.
\end{proposition}

\begin{proof}
We consider the automorphism of $R_\sigma [t,t^{-1}]$ that send $\chi^m$ to $t^ {\la (m) } \chi^m $.  This automorphism sends $\tilde{I} R_\sigma [t,t^{-1}]$ to $IR_\sigma [t,t^{-1}]$. Therefore, for each $c\in k^*$, we get an automorphism $\varphi_c: R_\sigma \to R_\sigma$ such that $\varphi_c(I_c) = I$. Thus, $\lct(I_c) = \lct(I)$ for all $c\in k^*$. Additionally, since $\varphi_c(\fm_x)=\fm_x$, we see each ideal $I_c$ is $\fm_x$ primary. 
Now, we may apply Proposition \ref{p:lctlsc} to see $\lct(I_0) \leq \lct(I)$. Since $\init_{>}(I) = I_0$, we are done.\end{proof}

\subsection{Proof of Theorem \ref{b}}

Theorem \ref{b} is a direct consequence of Proposition \ref{deg}. Note that the proof of Proposition \ref{deg} is directly inspired by the main argument in \cite{mus}.

\begin{proposition}\label{deg} Let $\ab$ be a graded sequence of $\fm_x$-primary ideals on $X_\sigma$. We have that 
\[
\lct( \init_>(\ab))^n \hs(\init(\ab)) \leq \lct(\ab)^n  \hs(\ab).
\]
\end{proposition}

\begin{proof}
We first note that
\[
\hs(
 \init_>(\ab)
 ) \coloneqq  \limsup_{m\to \infty} \frac{ \length( \cO_{X_\sigma,x}/  \init_>(\fa_m))}{m^n/n!}
 =
\limsup_{m\to \infty} \frac{ \length( \cO_{X_\sigma,x}/  \fa_m)}{m^n/n!}
=: 
\hs(
 \ab),\]
 where the second equality follows from Proposition \ref{deformlength} and the other two are definitional. By Proposition \ref{deformlct}, 
 \[
 \lct(  \init_>(\ab))\leq \lct( \ab). \]
The result follows. 
\end{proof}

\begin{proof}[Proof of Theorem \ref{b}]
Since $ {\Val_{X,x}}^{\toric} \subset \Val_{X,x}$, we have
\[
\inf_{v\in \Val_{X,x} }\nvol(v) \leq 
\inf_{v\in \Val_{x,x}^{\toric}} \nvol(v).\]
We proceed to show the reveres inequality. 
Note that
\begin{equation} \label{eq10} \inf_{v\in \Val_{X,x} }\nvol(v) = \inf_{\ab \, \fm_x \text{-primary}   } \lct(\ab)^n \hs(\ab)
=
\inf_{\substack{\ab \, \fm_x \text{-primary} \\ \text{monomial}   }} \lct(\ab)^n \hs(\ab)
,\end{equation} where the first equality is stated in Proposition \ref{relate} and the second follows from Proposition \ref{deg}.
The last infimum in Equation \ref{eq10}  is equal to
\[ 
\inf_{\substack{\fa \, \fm_x \text{-primary} \\ \text{monomial}   }}  \lct(\fa)^n \hs(\fa) 
\]
by Lemma \ref{details}.
Thus, it is sufficient to show that
\[
\inf_{v\in {\Val_{X,x}}^{\toric}} \nvol(v) 
\leq 
\inf_{\substack{\fa\, \fm_x \text{-primary} \\ \text{ monomial}}} \lct(\fa)^n \hs(\fa) . 
\] 

Let $\fa$ be an $\fm_x$-primary monomial ideal. Since $\fa$ is a monomial ideal, there exists a toric valuation $v^\ast \in \Val_{X,x}^{\toric}$ such that $v^\ast $ computes $\lct(\fa)$. (This follows from the fact that there exists a toric log resolution of $\fa$.) By Proposition \ref{tech}, 
\[
\nvol(v^\ast )\leq \lct(\fa)^n \hs(\fa), \]
and the proof is complete.
\end{proof}

\subsection{An Example of Non-Divisorial Volume Minimizer}\label{example}

Let $V$  denote $\PP^2$ blown up at a point. Note that $V$ is a Fano variety. 
The affine cone over $V$ with embedding $-K_V$ is isomorphic to the toric variety $ X_\sigma$ with torus invariant point $x$. 

\begin{figure}[h]
\includegraphics[height=2 in]{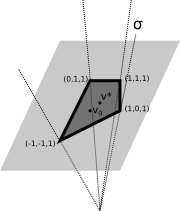}
 \caption{Drawn  is the cone $\sigma$. The toric variety $X_\sigma$ is isomorphic to the cone over $\PP^2$ blown up at a point.}
\end{figure}

We seek to find a minimizer of the function $\Val_{X_\sigma,x}^{\toric} \to   \RR_{>0}$ defined by 
$v_u \mapsto \nvol(v_u)$.
Since the normalized volume is invariant under scaling, it is sufficient to consider elements $u\in \Int(\sigma)$ of the form $u= (a,b,1) \in \Int(\sigma)$. 
We have  \[
A_{X_\sigma}(v_{(a,b,1)}) = \langle (a,b,1), (0,0,1) \rangle= 1.\]
The normalized volume of $v_{(a,b,c)}$ is
\[
\nvol( v_{(a,b,1)} ) \coloneqq  A_X(v_{(a,b,1)})^3 \vol(v_{(a,b,1)}) = 3 ! \Vol ( \sigma^\vee \cap H_{(a,b,1)}(1) ).\]
After computing the previous volume, we see that the function is minimized at 
\[
(a^\ast ,b^\ast ,1 ) = (4/3 -\sqrt{13}/3, 4/3 -\sqrt{13}/3, 1), \]
with $\nvol(v_{(a^\ast,b^\ast,1) } ) = \frac{1}{12} ( 46+13\sqrt{13}) $.  By Theorem \ref{b}, the toric volume minimizer $v^\ast = v_{(a^\ast ,b^\ast,1)}$ is also a minimizer of $\nvol_{X_\sigma,x}$. 

 Note that there cannot exist an additional volume minimizer of $\nvol_{X_\sigma,x}$ that is divisorial. Indeed, this follows from Proposition \ref{rational}.3.

\begin{appendix}

\section{Mutlplicities and Log Canonical Thresholds in Families}\label{secmult}

In this section we provide information on the behavior of the Hilbert-Samuel multiplicity and log canonical threshold in a family. The content of this section is well known to experts, but does not  necessarily appear in the literature in the follow form. 
 The following propositions will be useful in the proof of Proposition \ref{limits}.

\subsection{Multiplicities}

We first recall
a few notions from \cite[Section 14]{Mats}.
Let $(A,\fm)$ be a local ring of dimension $n$ and $\fa \subseteq A$ an ideal. We say that $\fb $ is a \emph{reduction} of $\fa$ if $\fb\subseteq \fa$ and there exists $r>0$ so that $\fb^{r}\fa = \fa^{r+1}$. If $\fa$ is an $\fm$-primary ideal and $\fb$ is a reduction of $\fa$, then $\hs(\fa)= \hs(\fb)$.

Assuming $A/\fm$ is an uncountable field and $\fa$ is an $\fm$-primary ideal, there always exists a parameter ideal $\fb\subset A$ such that $\fb$ is a reduction of $\fa$. Note that a \emph{parameter} ideal is an $\fm$-primary ideal generated by $n$ elements. 
If $\fb = (f_1,\ldots,f_n)$ is a parameter ideal, then
\[
\hs(\fb) = \sum_{i=0}^n (-1)^i  h_i \left(   \cK_\bullet \left( \underline{f} ,A  \right)     \right) .  \]
where the right hand sum is the Euler characteristic of the Koszul homology of $A$ with respect to the sequence $\underline{f} = f_1,\ldots,f_n$. As we see in the following proof, reducing to the parameter ideal case is useful in understanding multiplicities.

The following proposition is well known. Related statements appear in \cite{Lip}.

\begin{proposition}\label{cmult} 
Let $X$ and $T$ be varieties and $x\in X$ a closed point. 
If $\fa \subseteq \cO_{X\times T}$ is an ideal such that $\fa_t:=\fa \cdot \cO_{X\times \{ t \}}$ is $\fm_{x}$ primary for all closed points $t\in T$, then there exists an open set $U\subseteq T$
such that 
$\hs(\fa_t)$ is constant for all $t\in U$. 
\end{proposition}

\begin{proof}
We may assume that $X$ and $T$ are irreducible affine varieties. Let $\eta$ denote the generic point of $T$ and consider the ideal $\fa_{K(t)}$ on $X _{K(T)}$. Since $K(T)$ is an uncountable field (our base field $k$ is uncountable by assumption $k \subseteq K(T)$), we may find 
\[
b_1,\ldots, b_n\in \fa_\eta\]
such that $(b_1,\ldots, b_n)$ is a reduction of $\fa_\eta$. Now, choose an open subset $V\subseteq T$ such that each $b_i$ extends to an element $\tilde{b}_i \in \fa(X\times V)$. Let $\fb=(\tilde{b}_1,\ldots,\tilde{b}_n)\subseteq \cO_{X\times V}$. After shrinking $V$, we may assume that $\fb_t:= \fb \cdot \cO_{X\times \{t \}}$ is $\fm_{\{x\} }$-primary and a reduction of $\fa_t$ for all closed points $t\in V$. Therefore, $\hs(\fa_t) = \hs(\fb_t)$ for all $t\in V$.

We are reduced to showing that 
\[
t\in V \mapsto \hs(\fa_t) = \sum_{i} (-1)^i  h_i \left(   \cK_\bullet \left( \underline{b}_t , \cO_{p^{-1}(t),x}  \right)     \right) .
\]
is constant on an open set $U \subseteq V$. 
Note that there is a natural map
\[
 H_i \left ( \cK_\bullet \left( \underline{b} , \cO_{X\times T,\{x\} \times T}  \right) \right)\otimes k(t)
 \longrightarrow
 H_i \left ( \cK_\bullet \left( \underline{b}_t , \cO_{p^{-1}(t),x}  \right) \right)
,\]
and for $t= \eta$, the generic point of $T$, the map is an isomorphism. We choose an open set $U\subseteq V$ such that, for each $i=0,\ldots, n$, the above map is an isomorphism for all $t\in U$ and the dimension of $H_i \left ( \cK_\bullet \left( \underline{b} , \cO_{X\times T,\{x\} \times T}  \right) \right)\otimes k(t)$ is constant for all $t\in U$. 
 This completes the proof. 
\end{proof}

\subsection{Log Canonical Thresholds}

The following statements are well known, but do not explicitly appear in the literature. 

\begin{proposition}\label{clct} Let $X$ and $T$ be varieties such that $X$ is klt. Fix an ideal $\fa \subseteq \cO_{X\times T}$, and set $\fa_t: = \fa\cdot \cO_{X\times \{t\}}$. There exists a nonempty open set $U\subseteq T$ such that $\lct( \fa_t)$ is constant for all closed points $t\in U$.
\end{proposition}

\begin{proof}
Let $\mu: X' \to X\times T$ be a log resolution of $\fa$, and set $p' = p\circ \mu$:
\[
\begin{tikzcd}
X' \arrow{rr}{\mu} \arrow{rd}{p'} & & X \times T \arrow{ld}{p} \\
 & T&
\end{tikzcd}
.\]
Let $D$ be the divisor on $X'$ such that $ \fa \cdot \cO_{X'} = \cO_{X'}(-D)$ and $E_1,\ldots, E_r$ be the prime components of $\Exc(\mu)+D_{\red}$. 
After shrinking $T$, we may assume that each $E_i$ surjects onto $T$. 

We claim that on an open set $U\subset T$, $\mu_t: X'_t \to X_t$ is a log resolution of $\fa_t$ for all $t\in U$, where $X'_t \coloneqq  X'_{p^{-1}(t)}$ and $X_t := X\times \{t \}$. Indeed, by generic smoothness \cite[Corollar III.10.7]{Har} applied to $X'$, each $E_i$, and all the intersections of the $E_i$, we may find such a locus $U \subset T$. 

Now, we have $(K_{X'/ X\times T})\vert_{X_t} = K_{X'_t/X_t}$ and $\fa_t \cdot \cO_{X'_t} = 
\cO_{X'_t}(-D \vert _{X'_t} )$ for $t\in U$. Therefore, $\lct(\fa_t) = \min_{i=1,\ldots,r} \ord_{E_i}( K_{X'/X})/ \ord_{E_i}(D)$ for all closed points $t\in U$, and we are done. 
\end{proof}

\begin{proposition}\label{p:lctlsc}
Let $X$ be a klt variety, $T$ a smooth curve, and $t_0\in T$ a closed point.
Fix an ideal $\fa \subseteq \cO_{X\times T}$, and set $\fa_t:=\fa \cdot \cO_{X\times \{t\}}$. If the closed subscheme $V(\fa)\subseteq X$ is proper over $T$, then
there exists an open neighborhood $t_0 \in U \subseteq T$ such that 
\[
\lct( \fa_{t_0}) \leq \lct( \fa_t) \]
for all closed points $t\in U$. 
\end{proposition}
\begin{remark}
The condition that $V(\fa)$ is proper over $T$ holds in the following situation. Assume each ideal $\fa_t$ is $\fm_x$-primary for all closed points $t\in T$. Thus, $V(\fa)_{\red} = \{x \} \times T$, and $V(\fa)$ is proper over $T$. 
\end{remark}
\begin{proof}
By Proposition \ref{clct}, we may choose a nonempty open set $V\subseteq T$ such that 
$\lct(\fa_t)$ takes the constant value $\la$ for all $t\in V$. We will show $\lct(\fa_{t_0}) \leq \la$. Therefore, $U= V \cup \{t_0\}$ will satisfy the conclusion of our proposition. 

In order to prove $\lct(\fa_{t_0})\leq \la$, we will show  $V(\cJ(X\times T, \fa^\la))$ intersects $X\times \{t_0\}$. Since
$\cJ(X\times \{t_0\}, \la \cdot \fa_{t_0} )\subseteq\cJ(X\times T, \la \cdot \fa)\cdot \cO_{X\times \{t_0\}}$ \cite[Theorem 9.5.16]{Laz}, we will then conclude
$\cJ(X\times \{t_0\}, \la \cdot \fa_{t_0} )\neq \cO_{X\times \{t_0\}}$.

To  show $V(\cJ(X\times T,\la \cdot  \fa))$ intersects $X\times \{t_0\}$, we first note that the image of $V(\cJ(X,\times T, \la \cdot \fa)$ under the projection map $p:X\times T \to T$ is closed. Indeed, $V(\cJ(X,\times T, \la \cdot \fa)$ is contained in $V(\fa)_{\red}$, and the latter set is proper over $T$.  

Next, we apply \cite[Example 9.5.34]{Laz} to choose a nonempty open set $V' \subseteq T$ such that $\cJ(X\times \{t \} ,\la \cdot \fa_t) = \cJ(X, \la \cdot \fa) \cdot \cO_{X\times \{t \}}$ for all $t\in V'$. (While this result in \cite{Laz} is stated when $X$ is smooth, the statement extends to the klt case.) Since $\cJ(X\times \{t \}, \la \cdot \fa_t) \neq \cO_{X\times \{t\}}$ for all $t\in V$, we see $V(\cJ(X,\la \cdot \fa))$ intersects $X\times \{t\}$ for all $t\in V\cap V'$. Therefore, the projection of $V(\cJ(X, \la \cdot \fa))$ to $Z$ contains $V\cap V'$. Since the projection is closed in $Z$, it must be all of $Z$. 
\end{proof}

\section{Valuations Computing Log Canonical Thresholds of Graded Sequences}\label{appcomputing}

In \cite{JM}, the authors prove the existence of valuations computing log canonical thresholds of graded sequences of ideals on smooth varieties. We show the following generalization to the case of singular varieties. While the statement is known to experts, it does not explicitly appear in the literature. 

\begin{theorem}\label{computingtheorem}
If $X$ is a klt variety and $\ab$ a graded sequence of ideals on $X$ such that $\lct(\ab)<+\infty$, then there exists $v^\ast \in \Val_X$ computing $\lct(\ab)$.
\end{theorem}

The proof we give is similar in spirit to the proof of \cite[Theorem 7.3]{JM}, but also relies on results in \cite{BFFU}.

\begin{proposition}
If $X$ is a klt variety and $\ab$ a graded sequence of ideals on $X$, then $v \mapsto v(\ab)$
is a continuous function on ${\Val_X \cap \{ A_X(v) < +\infty\}}$. 
\end{proposition}

\begin{proof}
We reduce the result to the smooth case.  Take a resolution of singularities $Y\to X$ and write $\ab^Y$ for the graded sequence of ideals on $Y$ defined by $\fa_m^Y = \fa_m \cdot \cO_Y$. Thanks to \cite[Corollary 6.3]{JM}, the function $v \mapsto v(\ab^Y)$ is continuous on $\Val_Y \cap \{A_Y(v) < +\infty\}$. 

Now, note that the natural map $\Val_Y \to \Val_X$ is a homeomorphism of topological spaces and $v(\ab)= v(\ab^Y)$. Since $A_X(v) = A_Y(v) + v(K_{Y/X})$,
$A_X(v) < +\infty$ if and only if $A_Y(v)<+\infty$.  $v(\ab) = v(\ab^Y)$, the proof is complete. 
\end{proof}

Now, we recall some formalism from \cite{BFFU}. A \emph{normalizing subscheme} on $X$ is a (non-trivial) closed subscheme of $X$ containing $\Sing(X)$. If $N$ is a normalizing subscheme of $X$, we set
\[
\Val_X^N := \{ v\in \Val_X \, \vert \, v(\cI_N)=1 \}
.\]

\begin{proposition}
Let $X$ be a klt variety, $\ab$ a graded sequence of ideals on $X$, and $N$ a normalizing subscheme of $X$ such that $N$ contains the zero locus of $\fa_1$. 
\begin{enumerate}
\item The function $v\mapsto v(\ab)$ is bounded on $\Val_X^N$. 
\item For each $M\in \RR$, the set $\{A_X(v) \leq M\} \cap \Val_X^N$ is compact.
\item There exists $\epsilon >0$ such that $A_X(v) >\epsilon$ for all $v\in \Val_X^N$. 
\end{enumerate}
\end{proposition}
\begin{proof}
Statements (1) and (2) appear in \cite[Proposition 2.5]{BFFU}  and \cite[Theorem 3.1]{BFFU}, respectively. 
For (3), we use some formalism from \cite[Section 2]{BFFU}. Let $\pi:Y \to X$ be a good resolution of $N$ and 
\[r_\pi^N : \Val_X^N \to \Delta_\pi^N
\]
the continuous retraction map. Since $X$ is klt, it is clear that there exists $\epsilon>0$ such that $A_X(v)>\epsilon$ for all $v\in \Delta_\pi^N$. Now, for $v\in \Val_X^N$, we have $A_X(v) \geq A_X(r_\pi^N(v))$ and the proof is complete. 
\end{proof}

\begin{proof}[ Proof of Theorem \ref{computingtheorem}]
Let $N$ be the subscheme of $X$ defined by the ideal $\cI_{\Sing(X)} \cdot \fa_1$. Note that $N$ is a normalizing subscheme of $X$ and $N$ contains the zero locus of $\fa_1$. By the previous proposition, we may choose $B\in \RR$ and $\epsilon>0$ so that $v(\ab) < B$ and $A_X(v)>\epsilon$ for all $v\in \Val_X^N$. 

Note that
\[\lct(\ab) = \inf_{v\in \Val_X^N} \frac{A_X(v)}{v(\ab)}.\]
Indeed, consider $v\in \Val_X$ such that $A_X(v)/v(\ab)< +\infty$. 
Since $v(\ab)>0$, then $v(\fa_1)>0$ and, thus, $v(\cI_N) >0$. We see $w = (1/v(\cI_N))v \in \Val_X^N$ and $A_X(w)/w(\ab) = A_X(v) /v(\ab)$. 

Now, fix $L> \lct(\ab)$. If $v\in \Val_X^N$ and $A_X(v)/v(\ab) \leq L$, then 
\[
\epsilon <A_X(v) \leq L v(\ab) \leq L \cdot B .\]
Therefore, 
\[
\lct(\ab)  = \inf_{v \in W } \frac{A_X(v)}{v(\ab)}
,\]
where 
\[W = \Val_X^N \cap \{ A_X(v) \leq L \cdot B \} \cap \{ L v(\ab) \geq \epsilon  \}.\]
We claim that $W$ is compact.  Indeed, 
$\Val_X^N \cap \{ A_X(v) \leq L \cdot B \} $ is compact by the previous proposition. Since $v \mapsto v(\ab)$ is continuous on $\Val_X^N \cap \{ A_X(v) \leq L \cdot B \} $,  $W$ is closed in $\Val_X^N \cap \{A_X(v) \leq L \cdot B\}$, and, thus, compact as well. Since $v \mapsto A_X(v) /v(\ab)$ is lower semicontinuous on the compact set $W$, there exists $v^\ast \in W$ such that $A_X(v^\ast)/v^\ast(\ab)= \lct(\ab)$. 
\end{proof}

\end{appendix}


\begin{thebibliography}{CMSB02}


\bibitem[Blu16]{Blum}
H.~Blum.
\newblock \emph{On divisors computing mld’s and lct’s}.
\newblock \texttt{arXiv:1605.09662.v2}.
    
\bibitem[BdFFU15]{BFFU}
S. Boucksom, T. de Fernex, C. Favre and  S. Urbinati.
\newblock\emph{Valuation spaces and multiplier ideals on singular varieties}.
\newblock Recent Advances in Algebraic Geometry.
\newblock Volume in honor of Rob Lazarsfeld's 60th birthday, 29--51.
\newblock London Math. Soc. Lecture Note Series, 2015.

\bibitem[BFJ08]{BFJ08} 
S.~Boucksom, C.~Favre and M.~Jonsson.
\newblock \emph{Valuations and plurisubharmonic singularities}.
\newblock Publ. Res. Inst. Math. Sci.  \textbf{44}  (2008), 449--494.

\bibitem[BHJ16]{BHJ}
S.~Boucksom, T.~Hisamoto and M.~Jonsson. 
\newblock \emph{Uniform K-stability and asymptotics of energy functionals in K\"ahler geometry}.
\newblock \texttt{arXiv:1603.01026}.

\bibitem[Cut13]{Cut}
D.~Cutkosky.
\newblock\emph{Multiplicities associated to graded families of ideals}.
\newblock Algebra Number Theory \textbf{7} (2013), 2059--2083. 

\bibitem[dF013]{DF}
T.~de Fernex.
\newblock \emph{Birationally rigid hypersurfaces}.
\newblock Invent. Math.
\textbf{192} (2013), 533--566.

\bibitem[dFEM03]{DFEMR}
T.~de Fernex, L.~Ein and M.~Musta\c{t}\u{a}. 
\newblock \emph{Bounds for log canonical thresholds with applications to birational rigidity}.
\newblock Math. Res. Lett.
\textbf{10} (2003), 219--236.

\bibitem[dFEM04]{DFEM}
T.~de Fernex, L.~Ein and M.~Musta\c{t}\u{a}. 
\newblock \emph{Multiplicities and log canonical threshold}.
\newblock J. Algebraic Geom.
\textbf{13} (2004), 603--615.



\bibitem[dFEM10]{Smooth}
T.~de Fernex, L.~Ein and M.~Musta\c{t}\u{a}. 
\newblock\emph{Shokurov’s ACC conjecture for log canonical thresholds on smooth varieties}. 
\newblock Duke Math. J. \textbf{152} (2010), 93--114.

\bibitem[dFEM11]{Bound}
T.~de Fernex, L.~Ein and M.~Musta\c{t}\u{a}. 
\newblock\emph{Log canonical thresholds on varieties with bounded singularities}. 
\newblock In \emph{Classification of algebraic varieties}.
\newblock EMS Ser. Congr. Rep., pp. 221--257. 
\newblock Eur. Math. Soc., Z{\"u}rich, 2011.

\bibitem[dF09]{FM}
T.~de Fernex and M.~Musta\c{t}\u{a}.
\newblock \emph{Limits of log canonical thresholds}.
\newblock Ann. Sci. \'Ec. Norm. Sup\'er. (4)
\textbf{42} (2009), 491--515.


\bibitem[DEL00]{DEL}
J.-P.~Demailly, L.~Ein and R.~Lazarsfeld. 
\newblock \emph{A subadditivity property of multiplier ideals}.
\newblock Michigan Math. J. \textbf{48} (2000), 137--156.


\bibitem[Eis95]{Bud}
D. Eisenbud. 
\newblock \emph{Commutative Algebra.}
\newblock  Graduate Texts in Mathematics, 150. 
\newblock Springer-Verlag, New York, 1995.

\bibitem[Eis11]{Eis}
E. Eisenstein. 
\newblock \emph{Inversion of adjunction in high codimension}.
\newblock Ph.D.Thesis, University of Michigan, 2011.  

\bibitem[ELS03]{ELS}
L.~Ein, R.~Lazarsfeld and K.~E.~Smith.
\newblock \emph{Uniform approximation of Abhyankar valuation ideals in smooth function fields}.
\newblock Amer. J. Math. \textbf{125} (2003), 409--440.

\bibitem[EGA]{EGA}
A.~Grothendieck.
\newblock \emph{\'{E}l\'ements de g\'eom\'etrie alg\'ebrique}. 
Inst. Hautes \'Etudes Sci. Publ. Math., 1961.

\bibitem[FJ04]{FJ}
C.~Favre and M.~Jonsson.
\newblock \emph{The valuative tree.}
\newblock Lecture Notes in Mathematics, 1853. 
\newblock Springer-Verlag, Berlin, 2004. 
 

\bibitem[Ful93]{Fulton}
W.~Fulton.
 \newblock \emph{Introduction to toric varieties}.
\newblock Annals of Mathematics Studies, 131. 
 Princeton University Press, Princeton, NJ, 1993.

\bibitem[Har77]{Har}
R.~Hartshorne.
\newblock \emph{Algebraic geometry}.
\newblock Graduate Texts in Mathematics, 52. Springer-Verlag, New York-Heidelberg, 1977.

\bibitem[JM12]{JM}
M.~Jonsson and M.~Musta\c{t}\v{a}. 
\newblock \emph{Valuations and asymptotic invariants for sequences of ideals}.
\newblock Ann. Inst. Fourier \textbf{62} (2012), 2145--2209.

\bibitem[KK12]{Kaveh}
K.~Kaveh and A.~G.~Khovanskii.
\newblock\emph{Convex bodies and multiplicities of ideals}.
\newblock Proc. Steklov Inst. Math. \textbf{286} (2014), 268--284.

\bibitem[Kol08]{Which}
J.~Koll{\'a}r.
\newblock \emph{Which powers of holomorphic functions are integrable?}.
\newblock \texttt{arXiv:0805.0756}.

\bibitem[KM98]{KM}
J.~Koll{\'a}r and S.~Mori.
\newblock \emph{Birational geometry of algebraic varieties}.
\newblock Cambridge Tracts in Math., 134. 
Cambridge Univ. Press, Cambridge, 1998.

\bibitem[Kol97]{SingPairs}
J.~Koll\'ar. 
\newblock \emph{Singularities of pairs}. 
\newblock Algebraic geometry - Santa Cruz 1995, 221--287. 
\newblock Proc. Sympos. Pure Math., 62, Part 1, Amer. Math. Soc., Providence, RI, 1997.

\bibitem[Kol13]{K13}
J.~Koll{\'a}r.
\newblock \emph{Singularities of the minimal model program}.
\newblock Cambridge Tracts in Math., 200. 
Cambridge Univ. Press, Cambridge, 2013.

\bibitem[Laz04]{Laz} 
R.~Lazarsfeld.  
\newblock\emph{Positivity in algebraic geometry. I-II}.
\newblock Ergebnisse der Mathematik und ihrer Grenzgebiete, 48-49.
\newblock Springer-Verlag, Berlin, 2004.


\bibitem[LM09]{NO}
R.~Lazarsfeld and M.~Musta\c{t}\v{a}. 
\newblock\emph{Convex bodies associated to linear series}. 
\newblock Ann. Sci. {\'E}́c. Norm. Sup{\'e}r. (4), \textbf{42} (2009), 783--835.

\bibitem[Li15a]{Li1}
C.~Li.
\newblock\emph{Minimizing normalized volumes of valuations}.
\newblock \texttt{arXiv:1511.08164v3}.

\bibitem[Li15b]{Li2}
C.~Li.
\newblock\emph{K-semistability is equivariant volume minimization}.
\newblock \texttt{arXiv:1512.07205}.

\bibitem[LL16]{LiLiu}
C.~Li and Y.~Liu.
\newblock \emph{K{\"a}hler-Einstein metrics and volume minimization}.
\newblock \texttt{arXiv:1602.05094}.

\bibitem[LX16]{LiXu}
C.~Li and C.~Xu.
\newblock\emph{Stability of valuations and Koll{\'a}r components}.
\newblock \texttt{arXiv:1604.05398}.

\bibitem[Lip82]{Lip}
J.~Lipman.
\newblock \emph{Equimultiplicity, reduction, and blowing up}.
\newblock In \emph{Commutative algebra ({F}airfax, {V}a., 1979)}. 
\newblock Lecture Notes in Pure and Appl. Math. 68. 
Dekker, New York, 1982. 


\bibitem[Liu16]{Liu}
Y.~Liu.
\newblock\emph{The volume of singular K\"ahler-Einstein Fano varieties}.
\newblock \texttt{arXiv:1605.01034v2}.

\bibitem[MS06]{MDS}
D.~Martelli and J.~Sparks. 
\newblock \emph{Toric geometry, {S}asaki-{E}instein manifolds and a new infinite class of {A}d{S}/{CFT} duals}.
\newblock Comm. Math. Phsys. \textbf{262} (2006), 51--89.

\bibitem[Mat89]{Mats}
H.~Matsumura. 
\newblock \emph{Commutative ring theory}.
Cambridge Stud. Adv. Math., 8. 
Cambridge Univ. Press, Cambridge, 1989.

\bibitem[Mus02]{mus} 
M.~Musta\c{t}\v{a}. 
\newblock\emph{On multiplicities of graded sequences of ideals}.
\newblock J. Algebra \textbf{256} (2002), 229--249.

\bibitem[Tak06]{Tak}
S.~Takagi.
\newblock \emph{Formulas for multiplier ideals on singular varieties}.
\newblock Amer. J. Math. \textbf{128} (2006), 1345--1362.

\bibitem[Tak13]{Takpairs}
S.~Takagi.
\newblock \emph{A subadditivity formula for multiplier ideals associated to log pairs}.
\newblock Proc. Amer. Math. Soc. \textbf{141} (2013), 93--102.

\bibitem[Tei77]{Min}
B.~Teissier. 
\newblock \emph{
Sur une in{\'e}galit{\'e} {\`a} la Minkowski pour les multiplicit{\'e}s
}.   {A}ppendix to a paper of {D}. {E}isenbud and {H}.{I}. {L}evine, {T}he degree of a {C}$^\infty$ map germ. Ann. of Math. \textbf{106} (1977), 19--44.

\end{thebibliography}
\end{document}